\documentclass[12pt]{amsart}
\usepackage{comment}
\usepackage[utf8]{inputenc}
\usepackage[T1]{fontenc}
\DeclareMathAlphabet{\mathpzc}{OT1}{pzc}{m}{it}
\usepackage{geometry}

\usepackage{yfonts}
\usepackage[dvipsnames]{xcolor}
\usepackage{xstring}

\newcommand{\ja}[2][e]{%
    \IfStrEqCase{#1}{%
        {e}{\textcolor{Fuchsia}{\textbf{*Jason: #2*}}}%
        {c}{}%
    }
}
\newcommand{\gf}[2][e]{%
    \IfStrEqCase{#1}{%
        {e}{\textcolor{RoyalBlue}{\textbf{*Gary: #2*}}}%
        {c}{}%
    }
}
\newcommand{\cgt}[2][e]{%
    \IfStrEqCase{#1}{%
        {e}{\textcolor{Red}{\textbf{*Cecilia: #2*}}}%
        {c}{}%
    }
}
\newcommand{\sv}[2][e]{%
    \IfStrEqCase{#1}{%
        {e}{\textcolor{PineGreen}{\textbf{*Sandro: #2*}}}%
        {c}{}%
    }
}
\newcommand{\jaadd}[2][e]{%
    \IfStrEqCase{#1}{%
        {e}{\textcolor{Fuchsia}{\textbf{#2}}}%
        {r}{#2}%
    }
}
\newcommand{\gfadd}[2][e]{%
    \IfStrEqCase{#1}{%
        {e}{\textcolor{RoyalBlue}{\textbf{#2}}}%
        {r}{#2}%
    }
}
\newcommand{\cgtadd}[2][e]{%
    \IfStrEqCase{#1}{%
        {e}{\textcolor{Red}{\textbf{#2}}}%
        {r}{#2}%
    }
}
\newcommand{\svadd}[2][e]{%
    \IfStrEqCase{#1}{%
        {e}{\textcolor{PineGreen}{\textbf{#2}}}%
        {r}{#2}%
    }
}

\usepackage{amsfonts}
\usepackage{amssymb}
\usepackage{amsthm}
\usepackage{amsmath}
\usepackage{amscd}
\usepackage[shortlabels]{enumitem}
\usepackage{mathrsfs}
\usepackage{tikz}
\usetikzlibrary{calc,arrows,decorations.pathreplacing}
\usepackage{nicefrac, xfrac}
\usepackage{mathtools,xparse}

\usepackage[pagebackref, colorlinks = true,
linkcolor = red,
urlcolor  = blue,
citecolor = blue,
anchorcolor = blue]{hyperref}
\hypersetup{pdftitle={\@title},pdfauthor={\@author}}

\theoremstyle{plain}

\newtheorem{theorem}{Theorem}[section]

\newtheorem{lemma}[theorem]{Lemma}

\newtheorem{proposition}[theorem]{Proposition}

\theoremstyle{definition}
\newtheorem{definition}[theorem]{Definition}

\newtheorem{remark}[theorem]{Remark}
\newtheorem{example}[theorem]{Example}

\newtheorem{thmx}{Theorem}

\newtheorem*{theorem*}{Theorem}

\newcommand{\vp}{\varphi}
\renewcommand{\epsilon}{\varepsilon}

\newcommand{\sub}{\ensuremath{ \subseteq}} 
\newcommand{\set}[1]{\ensuremath{ \left\{#1\right\} }} 
\newcommand{\norm}[1]{\ensuremath{ \|#1\| }} 

\newcommand{\vertiii}[1]{{\left\vert\kern-0.25ex\left\vert\kern-0.25ex\left\vert #1
		\right\vert\kern-0.25ex\right\vert\kern-0.25ex\right\vert}}
\newcommand{\trinorm}[1]{\ensuremath{\vertiii{#1}}} 
\newcommand{\absval}[1]{\ensuremath{ \left\lvert#1\right\rvert }} 
\newcommand{\spot}{\ensuremath{ \makebox[1ex]{\textbf{$\cdot$}} }}

\newcommand{\ol}[1]{\overline{#1}}

\newcommand{\bs}{\ensuremath{\backslash}}

\newcommand{\lt}{\ensuremath{\left}}
\newcommand{\rt}{\ensuremath{\right}}

\newcommand{\Om}{\ensuremath{\Omega}}

\newcommand{\Gm}{\ensuremath{\Gamma}}

\newcommand{\Sg}{\ensuremath{\Sigma}}

\newcommand{\Dl}{\ensuremath{\Delta}}

\def\L{\ensuremath{\mathcal L}}
\newcommand{\om}{\ensuremath{\omega}}
\newcommand{\lm}{\ensuremath{\lambda}}
\newcommand{\gm}{\ensuremath{\gamma}}
\newcommand{\al}{\ensuremath{\alpha}}
\newcommand{\bt}{\ensuremath{\beta}}
\newcommand{\sg}{\ensuremath{\sigma}}
\newcommand{\ep}{\ensuremath{\epsilon}}

\newcommand{\kp}{\ensuremath{\kappa}}
\newcommand{\zt}{\ensuremath{\zeta}}
\newcommand{\ta}{\ensuremath{\theta}}

\newcommand{\vta}{\ensuremath{\vartheta}}
\newcommand{\vkp}{\ensuremath{\varkappa}}
\newcommand{\vrho}{\ensuremath{\varrho}}

\DeclareSymbolFont{bbold}{U}{bbold}{m}{n}
\DeclareSymbolFontAlphabet{\mathbbold}{bbold}

\newcommand{\ind}{\ensuremath{\mathbbold{1}}}
\DeclareMathOperator*{\esssup}{ess\sup}
\DeclareMathOperator*{\essinf}{ess\inf}

\DeclareMathOperator{\supp}{supp}

\newcommand{\cA}{\ensuremath{\mathcal{A}}}
\newcommand{\cB}{\ensuremath{\mathcal{B}}}

\newcommand{\cF}{\ensuremath{\mathcal{F}}}

\newcommand{\cH}{\ensuremath{\mathcal{H}}}

\newcommand{\cJ}{\ensuremath{\mathcal{J}}}

\newcommand{\cL}{\ensuremath{\mathcal{L}}}

\newcommand{\cP}{\ensuremath{\mathcal{P}}}

\newcommand{\cZ}{\ensuremath{\mathcal{Z}}}

\newcommand{\sA}{\ensuremath{\mathscr{A}}}
\newcommand{\sB}{\ensuremath{\mathscr{B}}}

\newcommand{\sF}{\ensuremath{\mathscr{F}}}

\newcommand{\CC}{\ensuremath{\mathbb C}} 

\newcommand{\EE}{\ensuremath{\mathbb E}}

\newcommand{\NN}{\ensuremath{\mathbb N}}

\newcommand{\PP}{\ensuremath{\mathbb P}}
\newcommand{\QQ}{\ensuremath{\mathbb Q}}
\newcommand{\RR}{\ensuremath{\mathbb R}}

\def\var{\text{{\rm var}}}
\def\Var{\text{{\rm Var}}}
\def\BV{\text{{\rm BV}}}
\def\Leb{\text{{\rm Leb}}}

\def\lt{\left}
\def\rt{\right}

\providecommand{\phantomsection}{}
\AtBeginDocument{\let\textlabel\label}
\makeatletter
\newcommand{\mylabel}[2]{\raisebox{.7\normalbaselineskip}{\phantomsection}(#1)%
	\def\@currentlabel{#1}\textlabel{#2}}
\makeatother

\makeatletter
\newcommand\xlabel[2][]{\phantomsection\def\@currentlabelname{#1}\label{#2}}
\makeatother

\newcommand{\bcomma}{,\allowbreak}

\ExplSyntaxOn
\NewDocumentCommand{\mathlist}{ O{,} m m }
 {
  \egreg_mathlist:nnn { #1 } { #2 } { #3 }
 }

\seq_new:N \l__egreg_mathlist_seq
\cs_new_protected:Npn \egreg_mathlist:nnn #1 #2 #3
 {
  \seq_set_split:Nnn \l__egreg_mathlist_seq { #1 } { #3 }
  \seq_use:Nnnn \l__egreg_mathlist_seq { #2 } { #2 } { #2 }
 }
\ExplSyntaxOff

\newcommand{\oio}{\ensuremath{\omega\in\Omega}}
\newcommand{\nin}{\ensuremath{n\in\NN}}

\allowdisplaybreaks

\numberwithin{equation}{section}
\title[]{Compound Poisson Statistics for 
dynamical systems via spectral perturbation }
\date{\today}

\author{Jason Atnip}
\address{School of Mathematics and Physics, The University of Queensland, St Lucia, QLD 4072, Australia}
\email{\href{j.atnip@uq.edu.au}{j.atnip@uq.edu.au} }

\author{Gary Froyland}
\address{School of Mathematics and Statistics, University of New South Wales, Sydney, NSW 2052, Australia}
\email{\href{g.froyland@unsw.edu.au}{g.froyland@unsw.edu.au} }

\author{Cecilia Gonz\'alez-Tokman}
\address{School of Mathematics and Physics, The University of Queensland, St Lucia, QLD 4072, Australia}
\email{\href{cecilia.gt@uq.edu.au}{cecilia.gt@uq.edu.au} }

\author{Sandro Vaienti}
\address{Aix Marseille Université, Université de Toulon, CNRS, CPT, 13009 Marseille, France}
\email{\href{vaienti@cpt.univ-mrs.fr}{vaienti@cpt.univ-mrs.fr} }

\begin{document}

	\maketitle
	
\begin{abstract}
We consider random transformations  $T_\omega^n:=T_{\sigma^{n-1}\omega}\circ\cdots\circ T_{\sigma\omega}\circ T_\omega,$ where each map $T_{\omega}$ acts on a complete metrizable space $M$.  The randomness comes from an invertible  ergodic driving map $\sigma:\Omega\to\Omega$ acting on a probability space $(\Omega,\cF,m).$  For a family of random target sets $H_{\omega, n}\subset M$ that shrink as $n\to\infty$, we consider quenched compound Poisson statistics of returns of random orbits to these random targets. 
We develop a spectral approach to such statistics: 
associated with the random map cocycle is a transfer operator cocycle $\mathcal{L}^{n}_{\omega,0}:=\mathcal{L}_{\sigma^{n-1}\omega,0}\circ\cdots\circ\mathcal{L}_{\sigma\omega,0}\circ\mathcal{L}_{\omega,0}$, where $\mathcal{L}_{\omega,0}$ is the transfer operator for the map $T_\omega$.
We construct a perturbed cocycle with generator $\mathcal{L}_{\omega,n,s}(\cdot):=\mathcal{L}_{\omega,0}(\cdot e^{is\mathbbold{1}_{H_{\omega,n}}})$ and an associated random variable $S_{\omega,n,k}(x):=\sum_{j=0}^{k-1}\mathbbold{1}_{H_{\sigma^j\omega,n}}(T_\omega^jx)$, which counts the number of visits to random targets in an orbit of length $k$.
Under suitable assumptions, we show that in the $n\to\infty$ limit, the random variables $S_{\omega,n,n}$  
converge in distribution to a compound Poisson distributed random variable. We provide several explicit examples for piecewise monotone interval maps in both the deterministic and random settings.

\end{abstract}

\tableofcontents

\section{Introduction}
In dynamical systems, return-time statistics describe the number of times that trajectories return to small neighbourhoods of the initial condition. In addition to their intrinsic interest, these statistics are heavily used in the study of extreme events. Indeed, the occurrence of such events can often be characterised as visits of trajectories of a dynamical system to a specific (small) region of the phase space, sometimes called a target set, or a hole.

This work investigates return-time statistics for various sufficiently chaotic dynamical systems, using compound Poisson random variables.
These random variables describe the sum of a Poisson-distributed amount of independent identically-distributed random variables. More precisely, a random variable $Z:M\to
\mathbb{N}$ is compound Poisson distributed if there exists a Poisson random variable $N$ and a sequence $X_1, X_2, \dots$ of non-negative i.i.d.\ random variables $X_k:M\to\mathbb{N}$ such that 
\begin{align}\label{def CP}
    Z = \sum_{k=1}^N X_k.
\end{align}
Compound Poisson random variables are useful to model the cumulative effect of random events which occur in a fixed amount of time, e.g. the total amount of rainfall in a year. 

A random dynamical system is defined through random compositions of maps $T_\omega:M\to M$, drawn from a collection $\{T_\omega\}_{\omega\in\Omega}$.
	An ergodic invertible measure-preserving map $\sigma:\Omega\to\Omega$ on a probability space $(\Omega,\cF,m)$ creates a map cocycle, or random dynamical system, driven by $\sigma$, namely $T_\omega^n:=T_{\sigma^{n-1}\omega}\circ\cdots\circ T_{\sigma\omega}\circ T_\omega$.
 This very general driving setup enables the study of random dynamics far beyond (and including)  deterministic maps (where $T_\om=T$ for every $\om\in\Omega$) and i.i.d. compositions of maps $T_\omega$ (where $\Omega$ is a set of bi-infinite sequences of symbols labelling the different maps, $\sigma$ is the left shift and $m$ is an infinite product measure $m=\nu^{\otimes \mathbb Z}$).

To investigate return-time statistics in the general setting of random dynamical systems, 
we consider a sequence of random target sets $H_{\omega, n}\subset M,$ which are decreasing in $n\in \mathbb N$ for each $\omega \in \Omega.$ 
For $\omega\in\Omega,$ $x\in M$, $n\ge 1$, and $k\le n$, we are interested in studying the following quantities
$$
S_{\om,n,k}(x):=\sum_{j=0}^{k-1}{\ind}_{H_{\sg^j\om,n}}(T_\om^jx),
$$
which count the number of times an  orbit of length $k$ hits the random sequence of targets on their associated $\omega$-fibers.

It is well known for deterministic dynamical systems, that the limit of $ S_{\om,n,k}$ for $n\rightarrow \infty,$ is degenerate unless one chooses $k$ {\em as a function} of the target set, which is known as a Kac scaling. In the context of random dynamical systems, there are different Kac-type scalings, see the discussion in Remark \ref{otsf}. We will use the scaling \eqref {qqss} introduced in our former work \cite{AFGTV-TFPF}, which means setting $k$ to be $n= \left\lfloor\frac{t_\omega+\xi_{\omega,n}}{\mu_{\omega,0}(H_{\omega, n)}}\right\rfloor$.
We therefore consider  the associated random variable $S_{\omega, n, n}:M\rightarrow \mathbb{N}$ on $M$ and  the discrete distributions
\begin{equation}\label{masdis}
\mu_{\omega,0}(\{x\in M: S_{\om,n,n}(x)=j\}),\qquad j\ge 0
\end{equation}
in the limit $n\rightarrow \infty,$
where $\{\mu_{\omega,0}\}_{\oio}$ is a family 
of probability measures on $M$
satisfying the equivariance condition $\mu_{\omega,0}\circ T_{\omega}^{-1}=\mu_{\sigma \omega,0}.$
These measures are a natural generalisation of invariant measures for a single deterministic map $T$ to the setting of random dynamical systems. 
Roughly speaking, our main result (Theorem~\ref{thm CF}) is that for $m$-a.e.\ $\omega$, in the $n\to\infty$ limit, the random variable  $S_{\omega,n,n}$  converges in distribution  to a random variable $Z$ that follows a compound Poisson distribution. 
\subsection{Background and notation}
In our previous paper \cite{AFGTV-TFPF} we developed a spectral approach for a quenched extreme value theory
that considers random dynamics on the unit interval with general ergodic invertible driving, and random observations. An extreme value law was derived using the first-order approximation of the leading Lyapunov multiplier  of a suitably perturbed transfer operator defined by the introduction of small random holes in a metric space $M.$ We were inspired by a result of Keller and Liverani \cite{keller_rare_2009} which, in the deterministic setting,  developed abstract conditions on the transfer operator $\cL$ and its perturbations $\cL_{\epsilon}$ (for each $\ep>0$)
to ensure good first-order behaviour with respect to the perturbation size. 

Our first task was to generalize the Keller-Liverani result when we have a sequential compositions of linear operators $\mathcal{L}_{\omega,0}^n:=\mathcal{L}_{\sigma^{n-1}\omega,0}\circ\cdots\circ \mathcal{L}_{\sigma\omega,0}\circ \mathcal{L}_{\omega,0}$, 
	where $\sigma:\Omega\to\Omega$ is an invertible, ergodic map on a configuration set $\Omega$.
We then consider a family of perturbed cocycles $\mathcal{L}_{\omega,\ep}^n:=\mathcal{L}_{\sigma^{n-1}\omega,\ep}\circ\cdots\circ \mathcal{L}_{\sigma\omega,\ep}\circ \mathcal{L}_{\omega,\ep}$, for each $\ep>0$, where the size of the perturbation $\mathcal{L}_{\omega,0}-\mathcal{L}_{\omega,\ep}$ is quantified by the value $\Delta_{\omega,\epsilon}=\nu_{\sg\om,0}\left((\cL_{\om,0}-\cL_{\om,\ep})(\phi_{\om,0})\right),$ where $\phi_{\om,0}$ and $\nu_{\om,0}$ (the conformal measure), are respectively the random eigenvector of $\cL_{\om,0}$
and of its dual with common eigenvalue $\lambda_{\omega,0}.$ 
	We obtained  an abstract quenched formula  for the Lyapunov multipliers $\lambda_{\omega,0}$ up to first order in the size of the perturbation $\Delta_{\omega,\epsilon}.$  	This map cocycle generates a transfer operator cocycle $\mathcal{L}_{\omega,0}^n,$ where $\mathcal{L}_{\omega,0}$ is the transfer operator for the map $T_\omega$.
	For each $\om\in\Om$ and each $\ep>0$, a random \jaadd[r]{target, referred to as a hole in \cite{AFGTV-TFPF}, $\~H_{\omega,\ep}\subset M$} is introduced; this allows us to define the perturbed transfer operator	$\mathcal{L}_{\omega,\epsilon}$  for the open map $T_\omega$ and \jaadd[r]{target $\~H_{\omega,\epsilon},$  namely $\mathcal{L}_{\omega,\epsilon}(f)=\mathcal{L}_\omega(\ind_{M\setminus \~H_{\omega,\epsilon}}f)$. }
 
 Suppose now $h_{\omega}:M\to\RR$ is a continuous function for each $\omega$ and write $\overline{z}_{\omega}$ as its essential supremum. For any $z_{\omega, N}<\overline{z}_{\omega}$ we can now define the set 
$$
\cH_{\om,z_{\omega, N}}:=\set{x\in M :h_\om(x)-z_{\omega, N}>0}, 
$$
which could be identified as a \jaadd[r]{target} in the space $M.$  The suffix $N$  for the point $z_{\omega, N}$ means that we will now consider $N$ such \jaadd[r]{target}s in order to study the {\em distribution of non-exceedances}:
 \begin{align}\label{D1}
&\mu_{\omega,0}\left(\lt\{x\in M: h_{\sigma^j\omega}(T^j_{\omega}(x))\le z_{\sigma^j\omega, N}, j=0, \dots, N-1\rt\}\right)	
\nonumber\\
&\qquad= 
\mu_{\om,0}\left(\set{x\in\cJ_{\om,0}:T_\om^j(x)\notin \cH_{\sg^j\om,z_{\sg^j\om,N}} \text{ for }j=0,\dots, N-1}\right)
	 \end{align}
 where $\mu_{\omega, 0}$ is the equivariant measure for the unperturbed system. Moreover we will require an asymptotic behavior for the \jaadd[r]{target}s of the type
 \begin{equation}\label{qqss}
     \mu_{\omega,0}(\cH_{\om,z_{\omega, N}})=(t_\omega+\xi_{\omega,N})/N,
     \end{equation}
     for a.e.\ $\omega$ and each $N\ge 1$, where $t_{\omega}$ is a positive random variable and $\xi_{\omega,N}$ goes to zero when $N\rightarrow \infty$; see Section \ref{sec: random setting} for more details. Following the spectral approach of \cite{keller_rare_2012}, we proved in \cite{AFGTV-TFPF} that the distribution (\ref{D1}) behaves asymptotically as $\sfrac{\lm_{\om,\ep_N}^N}{\lm_{\om,0}^N},$ where $\lm_{\om,\ep_N}$ is the Lyapunov multiplier of $\cL_{\omega, \epsilon_N.}$ Such a multiplier is obtained by the first order perturbation of $\lm_{\om,0}$ quoted above, and ultimately will produce the limit Gumbel's law
 $$
\lim_{N\to\infty} \mu_{\omega,0}\left(x\in M: h_{\sigma^j\omega}(T^j_{\omega}(x))\le z_{\sigma^j\omega, N}, j=0, \dots, N-1\right)
		=\exp\left(-\int_\Omega t_\omega\theta_{\omega,0}\ dm(\omega)\right),
 $$
where the {\em extremal index} $\theta_{\omega,0}$	 is given by the limit 
\jaadd[r]{$\theta_{\omega,0}=\lim_{N\to \infty}\frac{\lm_{\om,0}-\lm_{\om,\ep_N}}{\Dl_{\om,\ep_N}}.$}

There is another equivalent interpretation of Gumbel's law in terms of hitting times that will be useful to introduce the main topic of this paper.  
Let us now consider a general family of random \jaadd[r]{target}s $\{H_{\om, n}\}_{\om\in\Om},$ satisfying $H_{\om,n'}\sub H_{\om,n}$, $n'\geq n$, and $\lim_{n\to 0}\mu_{\om,0}(H_{\om,n})=0.$  The first (random) hitting time to a \jaadd[r]{target}, starting at initial condition $x$ and random configuration $\omega,$ is defined by:
$$
\tau_{\om, H_{\om,n}}(x):=\inf\{k\ge 1, T^k_{\om}(x)\in H_{\sigma^k\om, n}\}.
$$
Under the assumptions which allowed us to get Gumbel's law, in particular (\ref{qqss}) $
\mu_{\omega,0}(H_{\om,n})=\frac{t_\omega+\xi_{\omega,n}}{n},$ 
with $\xi_{\omega,n}$ going to zero when $n\rightarrow \infty,$  we can also prove that 
$$\lim_{n\to\infty}\mu_{\om,0}\left(\tau_{\om, H_{\om,n}} \mu_{\om,0}(H_{\om, n})>t_{\om}\right) = \exp\left(-{\int_\Om t_{\om} \theta_{\om,0}\ dm(\omega)}\right).
	$$
 This result suggests that the exponential law given by the extreme value distribution describes the time between  successive events in a Poisson process.

 \subsection{Main result and literature review}

Given $\omega\in \Omega$ and a sequence of random target sets $H_{\omega, n}\subset M,$
we consider the random variables
 \begin{equation}\label{ffss}
 S_{\om,n,n}(x):=\sum_{j=0}^{n-1}{\ind}_{H_{\sg^j\om,n}}(T_\om^jx),
 \end{equation}
 where the initial condition $x\in M$ is distributed according to $\mu_{\omega, 0}$. 
The main result of this work is that for an appropriate choice of target sets, the distribution of the random variables \eqref{ffss}
converge to a compound Poisson distribution as $n\rightarrow \infty.$ 
In order to establish this result, the measure of the target sets should follow the scaling (\ref{qqss}). Our proof relies on computing 
the characteristic function of the random variable (\ref{ffss}) and showing that, as $n\rightarrow \infty,$ 
this characteristic function converges pointwise to the characteristic function of a random variable which is discrete and infinitely divisible, and therefore compound Poisson distributed. 
Our main result could be summarised as follows. See Theorem~\ref{thm CF} for a detailed statement.
\begin{thmx}\label{main thm CF}
    For the  random perturbed system introduced above and  satisfying assumptions \eqref{C1}--\eqref{C8} and \eqref{S} (see Section \ref{sec: random setting} for full details), we have that for each $s\in \mathbb R$ 
    and $m$-a.e. $\oio$
    \begin{align} \label{ft}
		\lim_{n\to\infty}\mu_{\om,0}\lt(e^{is S_{\om,n,n}}\rt)
		=
		\lim_{n\to\infty}\frac{\lm_{\om,n,s}^n}{\lm_{\om,0}^n}
		=
		\exp\left(-(1-e^{is})\int_\Om t_\om\ta_\om(s)\, dm(\om)\right).    
    \end{align}
\end{thmx}
   The quantity $\lm_{\om,0}$ is the top Lyapunov multiplier of the operator $\cL_{\omega, 0},$  while $\lm_{\om,n,s}$ is the top multiplier of a perturbed operator $ \cL_{\om,n,s}(f):=\cL_{\om,0}(f\cdot e^{is\ind_{H_{\om,n}}})$.
   The quantities $\lambda_{\om,0}^n$ and $\lambda_{\om,n,s}^n$ are products of $n$ such multipliers along the $\omega$-orbit of length $n$.
   Notice that as soon as we have the limit characteristic function on the right hand side of (\ref{ft}), we could use L\'evy's inversion formula to get the mass distribution of \eqref{ffss}, namely
   $$
   \lim_{n\rightarrow \infty}\mu_{\omega,0}(S_{\omega, n,n}=k)=\lim_{T\to\infty}\frac{1}{2T}\int_{-T}^Te^{-isk}\exp\left(-(1-e^{is})\int_\Om t_\om\ta_\om(s)\, dm(\om)\right) ds.
   $$
   To obtain the limit \eqref{ft} we again use our generalized version of the Keller-Liverani theorem, with the associated extremal index $\ta_{\om}(s)$ which is given by
   $$
   \lim_{n\to\infty}
		\frac{\lm_{\om,0}-\lm_{\om,n,s}}{\lm_{\om,0}\mu_{\om,0}(H_{\om,n})}
		=(1-e^{is})\ta_{\om}(s).
  $$
  
It is not the first time that the Keller-Liverani perturbation theorem has been used to get the Poisson distribution for the number of visits in small sets. We quote here the result  \cite{Zhang2020} which holds in the deterministic setting and which exhibits  several differences when compared to our approach. First of all \cite{Zhang2020} computed the Laplace transform  of the random variable counting the number of visits in a decreasing sequence of sets with bounded cylindrical lengths around a non-periodic point. In the limit of  vanishing measure of the target set, the Laplace transform converges to the Laplace transform of the usual Poisson distribution, which from now on will be referred to as {\em standard Poisson}. This result does not cover the case when the target set is around a periodic point or when it has a different geometrical shape. Furthermore, even if the limiting Laplace transform could in principle be computed, it is not clear how to invert it in order to get the probability mass distribution. 

In Section \ref{sec deterministic} of this paper we will sketch how to apply our perturbation scheme to deterministic systems and we will show that for the aforementioned case of the target set around periodic points, we will get, as expected, the P\'olya-Aeppli distribution. We note that if the 
target sets are balls around a point $z$, then we have a dichotomy regarding the convergence of the distribution of the number of visits
for systems with a strong form of decay of correlations. Either $z$ is periodic and in that case we have convergence to a P\'olya-Aeppli distribution, or $z$ is not periodic and in
that case we have  convergence to a standard Poisson process, see \cite{AFV} for a rigorous proof of this claim for systems which exhibit correlation decay with respect to $L^1$ observables.  Other compound Poisson distributions will emerge in the random setting, where the notion of periodicity is lost, and we will give a few examples of them in Section~\ref{S:examples}. We will see that it is easier to construct examples that are not standard Poisson and this could be interesting for application to the real world, where the noise is a constituent of the environment, see \cite{CFVY} for an application to climate time series.

For quenched random dynamical systems, there \cgtadd[r]{are already two contributions proving convergence to a compound Poisson distribution. The first one}, \cite{FFMV20},  is based on  analogous results in the deterministic setting due to A-C Freitas, J-M Freitas and M. Magalhes \cite{
FFM18}.  What is actually shown is  the convergence of marked point processes for  random dynamical systems given by fibred Lasota Yorke maps. The technique is quite different from our spectral approach,  even though it uses  the Laplace transform to establish the convergence of random measures. In \cite{FFMV20}, our threshold assumption (\ref{qqss}) is replaced by the H\"usler type condition 
\begin{equation}\label{Hascon}
\sum_{j=0}^{N-1}\mu_{\omega,0}\left(h_{\sigma^j\omega}(T^j_{\omega}(x))> z_{\sigma^j\omega, N}\right)\rightarrow t,
\end{equation}
which has been used to deal with non-stationary extreme value theory, see \cite{H86, FFV017, FFMV016}, and
is proved under strong mixing requirements of the driving system guaranteeing at least polynomial decay of correlations for the marginal measure $\tilde{\mu}=\int \mu_{\omega,0}\ dm$.  In contrast, in this present work  
we instead only need that the driving system be ergodic. Moreover the unperturbed transfer operator is defined with the usual geometric potential,  forcing all of the fiberwise conformal measures to be equal to the Lebesgue measure. Finally, the only example completely treated gives convergence to a standard Poisson distribution. Recently another paper dealt with quenched Poisson processes for random subshifts of finite type, \cite{CS20}. It is proved  that hitting times to dynamically defined cylinders converge to a Poisson point process under the law
of random equivariant measures with super-polynomial decay of correlations. 
\cgtadd[r]{The second  
recent paper on quenched systems \cite{AHV24}, \jaadd[r]{which became available after the completion of the present work,} 
adapts  
the deterministic framework of \cite{HV20} to the random setting. 
The approach of \cite{AHV24} is based on a  probabilistic block-approximation, which does not require exponential decay of correlations of the equivariant measure $\mu_{\omega, 0}$.  Instead, \cite{AHV24} requires that the correlations decay with some polynomial rate together with the correlations of the marginal measure $\tilde{\mu}$, see above. \jaadd[r]{The allowance of polynomial decay} requires the stronger assumption that the driving of the base dynamical system $\sigma:\Om\to \Om$ must be mixing, \jaadd[r]{while we only require ergodicity in the base. In addition,} 
the threshold assumption of \cite{AHV24} fixes the time of observation by averaging over the measures of the target sets, and is therefore  a type of annealed version of our scaling (\ref{qqss}), see Remark \ref{otsf}. The convergence to the limit distribution is 
obtained along subsequences that require a certain kind of summability. \jaadd[r]{Despite the difference of assumptions,} all of the examples in \cite{AHV24} are similar to ours.}

In the deterministic framework, besides  \cite{FFM18}, two other recent papers developed compound Poisson statistics for general classes of dynamical systems, \cite{HV20} and \cite{GHV22}. Both papers aimed to compare a given probability measure -- namely the distribution of the number of visits to a set 
-- to a compound Poisson distribution.  This will give  an error for the total variation distance between the two distributions. Any compound Poisson distribution depends upon a set of parameters $\lambda_l, l\ge 1.$ It has been shown in \cite{HV20}, that those parameters are related to another sequence $\alpha_l, l\ge 1,$  which are the limits of the distribution of higher-order returns and therefore are in principle computable.  Whenever those limits  exist and the $\al_l$ 
verify a summability condition, the error term can be evaluated in two possible ways.  In \cite{HV20}, inspired in part by \cite{ChCo13}, 
the error term is evaluated using a very general approximation theorem that allows one to measure how close a return-time distribution is to  a compound binomial distribution, which in the limit converges to a compound
Poisson distribution. 
In \cite{GHV22} the error term is instead evaluated by an adaption of the classical Stein-Chen method \cite{CHS}. The two approaches target different classes of dynamical systems. The approach of \cite{HV20} is more geometric and is adapted to differentiable systems which are not necessarily exponentially mixing, whereas the approach of \cite{GHV22} is more devoted to symbolic and $\phi$-mixing systems.

The history of Poissonian distributions for the number of visits in small sets dates back to the seminal papers by Pitskel (1991)\cite{P91}  and Hirata
(1993) \cite{H93}.  There have been several other contributions employing different techniques; we provide here a non-exhaustive list: \cite{Abadi2001, HV2009, HHSSVV, KR14, KY18, HPS14, CHAZCOLL, Zwei16, Zwei17}. 
A complementary approach to the statistics
of the number of visits has been developed in the framework of extreme value theory, where
it is more often called point process, or particular kinds of it as the marked point process; besides the papers quoted above \cite{FFMV20, FFM18}, see also \cite{FRE13, AFV}. The distribution
of the number of visits to vanishing balls has been studied for systems modeled by a Young
tower in \cite{ChCo13, PESA16, HAWA16, HAYA17, YANG21}, and for uniformly hyperbolic systems in \cite{CFFHN} and \cite{AHV}.  Recurrence in billiards
provided recently several new contributions; for planar billiards in \cite{PESA10,  FRE14, CarneyNicolZhang, CarneyHollandNicol}  
and in \cite{pene2020spatio}  spatio-temporal Poisson processes
were obtained by recording not only the successive times of visits to a set, but also the
positions.

\section{The Deterministic Setting: A Motivational Approach}\label{sec deterministic}
While our results in the random setting, presented in Section \ref{sec: random setting}, imply the compound Poisson statistics for the deterministic setting, 
in order to motivate our results for random dynamical systems, we first give a sketch of our approach to obtain compound Poisson statistics in the deterministic setting via Keller-Liverani perturbation theory. 
We will also present a few general considerations that can be immediately translated to the random setting.

Suppose that \cgtadd[r]{$T:I\to I$ is 
a piecewise smooth ($C^{1+\alpha}$), piecewise expanding map with finitely many branches.
Let $\mu$ be an absolutely continuous invariant measure for $T$.  Consider a decreasing sequence of \jaadd[r]{(non-random) target}s $H_n$ shrinking to a finite set.}  
We aim to compute the following distribution
$$
\lim_{n\rightarrow \infty}\mu\left(\sum_{i=0}^{k-1}\ind_{H_n}(T^ix)=j\right)
$$
for $j=0,1,\dots$, where, as usual, we set $k=n$ and assume the scaling
\begin{align}\label{eq det scaling}
n={\lt\lfloor\frac{t}{\mu(H_n)}\rt\rfloor}.
\end{align}
Rather than following the approach of Zhang \cite{Zhang2020}, who computed the moment generating function, we compute the characteristic function (CF)  of the sum 
\begin{equation}
\label{Snndet}
S_{n,k}(x):=\sum_{i=0}^{k-1}\ind_{H_n}(T^ix).
\end{equation}
We show that it converges to the CF of a random variable (RV) $Z$, which has a compound Poisson distribution. We let $\cL$ denote the Perron-Frobenius operator acting on the space $\BV$ of complex-valued bounded variation functions, and for each $s\in\RR\bs\{0\}$ and $n\in\NN$ we define the perturbed operator
$$
\L_{n,s}(f)=\L(e^{is\ind_{H_n}}f).
$$ 
Note that if $s=0$,  then $\L_{n,0}=\L$ for all $n\in\NN$. 
Iterates of this operator are given by 
$$
\L_{n,s}^k(f)=\L^k(e^{isS_{n,k}}f)
$$
for each $n\geq 1$. If $h$ is the density of $\mu$ (with respect to Lebesgue), then 
$$
\int e^{isS_{n,k}} h\ dx=\int \L^k(e^{isS_{n,k}} h)\ dx=\int \L_{n,s}^k(h)\ dx.
$$
As usual we suppose that the operators $\L$ and $\L_{n,s}$ are quasi-compact, and in particular,
$$
\L_{n,s}(\cdot)=\lambda_{n,s}\nu_{n,s}(\cdot)\phi_{n,s}+Q_{n,s}(\cdot),
$$
where the above objects are deterministic versions of the corresponding random objects in (\ref{C3}).
As the \jaadd[r]{target}s $H_n$ shrink to a finite set, the operator $\L_{n,s}$ approaches $\L.$ The triple norm difference between the two operators is given by
\begin{equation}\label{rt}
|||(\L-\L_{n,s})|||\le |1-e^{is}|\mu(H_n)\le 2\mu(H_n).
\end{equation}
In Keller-Liverani \cite{keller_rare_2009} they consider the following normalizing quantity
\begin{align}\label{eq Delta formula}
\Delta_{n,s}=\int (\L-\L_{n,s})h\ dx=(1-e^{is})\mu(H_n).
\end{align}
In \cite[Section 4]{keller_rare_2012}, using the Keller-Liverani perturbation theory of \cite{KL99}, Keller shows that the assumptions of \cite{keller_rare_2009} hold (namely (A1)-(A6) in \cite{keller_rare_2009} or (1)-(6) in \cite{keller_rare_2012}) for a similar setting.\footnote{In fact Keller shows the assumptions (A1)-(A6) of \cite{keller_rare_2009} (the deterministic versions of our assumptions \eqref{C1}--\eqref{C7}) hold for further examples of shift maps and higher dimensional maps for which our theory also applies.}
Thus, applying Theorem 2.1 of \cite{keller_rare_2009}, we have
\begin{align}\label{def det theta}
    \lim_{n\to\infty}\frac{1-\lm_{n,s}}{\Dl_{n,s}}=1-\sum_{k=0}^\infty q_k(s) =:\ta(s),
\end{align}
where 
\begin{align}\label{eq thm KL09}
    q_k(s)&=\lim_{n\rightarrow \infty}\frac{1}{(1-e^{is})\mu(H_n)}\int (\L-\L_{n,s})\L^k_{n,s}(\L-\L_{n,s})(h)\ dx,
\end{align}
and the limit in \eqref{eq thm KL09} is assumed to exist. 
Calculating $q_k$ yields 
\begin{align}
    q_k(s)&=\lim_{n\rightarrow \infty}\frac{1}{(1-e^{is})\mu(H_n)}\int (\L-\L_{n,s})\L^k_{n,s}(\L-\L_{n,s})(h)\ dx
    \nonumber\\
    &=
    \lim_{n\rightarrow \infty}\frac{1}{(1-e^{is})\mu(H_n)}\int (1-e^{is\ind_{H_n}(T^{k+1}(x))})e^{is[\ind_{H_n}(T(x))+\dots +\ind_{H_n}(T^{k}(x))]}(1-e^{is\ind_{H_n}(x)})h(x)\ dx
    \nonumber\\
    &=\lim_{n\rightarrow \infty}\frac{1}{(1-e^{is})\mu(H_n)}\int_{H_n\cap T^{-(k+1)}(H_n)} (1-e^{is})^2e^{is[\ind_{H_n}(T(x))+\dots +\ind_{H_n}(T^{k}(x))]}h(x)\ dx
    \nonumber\\
    &=\lim_{n\rightarrow \infty}\frac{(1-e^{is})}{\mu(H_n)}\int_{H_n\cap T^{-(k+1)}(H_n)} e^{is[\ind_{H_n}(T(x))+\dots +\ind_{H_n}(T^{k}(x))]}h(x)\ dx.
    \label{cuc}
\end{align}
In view of (\ref{cuc}), we see that the orbit $T^l(x)$ can return to $H_n$ for $\ell=1,\dots,k.$ 
Each return adds a multiplicative factor $e^{is},$ for a positive $\ell.$ Therefore, for $\ell=0,\dots,k$, we define
$$
\beta_n^{(k)}(\ell):=\frac{\mu(x; x\in H_n, T^{k+1}(x)\in H_n, \sum_{j=1}^k\ind_{H_n}(T^jx)=\ell)}{\mu(H_n)}.
$$
Now suppose that the limit
$$
\beta_k(\ell):=\lim_{n\rightarrow \infty}\beta_n^{(k)}(\ell)
$$
exists.
Then we have the following alternate formulation of the $q_k$, 
\begin{align}
q_k(s) &=\lim_{n\rightarrow \infty}\frac{1}{1-e^{is}}\sum_{\ell=0}^k(1-e^{is})^2e^{i\ell s}\beta_n^{(k)}(\ell)
=(1-e^{is})\sum_{\ell=0}^ke^{i\ell s}\beta_k(\ell).
\label{eq alt q}
\end{align}
It follows from the assumption (A3) of \cite{keller_rare_2009} (the deterministic version of our assumption \eqref{C4}) that the sum 
\begin{align}\label{eq beta sum}
    \Sg:=\sum_{k=0}^{\infty}\sum_{\ell=0}^k\beta_k(\ell)
\end{align} 
converges absolutely\footnote{For more details in the random setting, see the arguments presented in Section \ref{sec: random setting}.}. Thus, combining \eqref{def det theta} and  \eqref{eq alt q} we obtain the following alternate expression for $\ta(s)$: 
\begin{align}\label{eq alt theta}
\theta(s)=1-(1-e^{is})\sum_{k=0}^{\infty}\sum_{\ell=0}^ke^{i\ell s}\beta_k(\ell).
\end{align}
In view of \eqref{eq thm KL09}, and using \eqref{eq det scaling} and \eqref{eq Delta formula}, we have 
$$
1-\lambda_{n,s}\approx \theta(s) \Delta_{n,s} =\theta(s) (1-e^{is})\frac{t}{n}.
$$
Following the approach of \cite{keller_rare_2012}, exponentiating the previous formula and taking $n\to\infty$, we 
obtain the following theorem. 
\begin{theorem}
\label{thm deterministic CF}
    Assume that the scaling \eqref{eq det scaling} holds and that conditions (A1)--(A7) of \cite{keller_rare_2009} hold for the operators $\cL_{n,s}$ for each $s\in\RR$ and all $n\in\NN$ sufficiently large. Then we have
    \begin{equation}\label{as}
        \lim_{n\rightarrow \infty}\int e^{isS_{n,n}} h\  dx=e^{-\theta (s) (1-e^{is}) t}=:\vp(s).
    \end{equation}
\end{theorem}

Taking \eqref{eq alt theta} together with the absolute convergence of the series \eqref{eq beta sum}, we see that 
 $\vp(s)$ is  continuous in $s=0,$ and therefore is the characteristic function of some random variable $Z$ on some probability space $(\Gamma', \mathcal{B}', \mathbb{P}')$ to which the sequence of random variables (using (\ref{eq det scaling}) and (\ref{Snndet}))
\begin{equation}\label{mil}
Z_n:=\sum_{i=0}^{{\lt\lfloor\frac{t}{\mu(H_n)}\rt\rfloor}}\ind_{H_n}\circ T^i\end{equation}
converges in distribution; this follows from 
 L\'evy Continuity Theorem (see \cite[Theorem 3.6.1]{lukacs}).  We let $\nu_Z$ denote the distribution of $Z,$ and 
 we now denote the invariant measure $\mu$ by $\PP$. 
It follows from the Portmanteau Theorem that the variable $Z$ is non-negative and integer valued as the distributional limit of a sequence of integer-valued RV\footnote{If $\nu_n$ denotes the distribution of $Z_n$, then the Portmanteau Theorem implies that 
$
    1= \limsup_{n\to\infty}\nu_n(\NN_0)=\nu_Z(\NN_0), 
$
where $\NN_0=\NN\cup\{0\}$.
}. 
Moreover it  is a standard result that  
\begin{align}\label{eq P Zn}
\mathbb{P}'(Z=k)=\nu_{Z}(\{k\})=\lim_{n\rightarrow \infty}\mathbb{P}(Z_n=k).
\end{align}
Note also that $Z$ is clearly infinitely divisible, since $e^{-\theta (s) (1-e^{is}) t}=(e^{-\theta (s) (1-e^{is}) t/N})^N$. This along with the fact that $Z$ is non-negative and integer valued  imply that  $Z$ has a compound Poisson distribution, see for instance \cite[p. 389]{univ} or \cite[Section 12.2]{fellerv1}. 
To obtain the probability mass function for $Z$ we can apply the L\'evy inversion formula to get 
\begin{align}\label{eq prob Z}
\nu_{Z}(\{k\})=\lim_{T\rightarrow \infty}\frac{1}{2T}\int_{-T}^{T}e^{-isk}e^{-\theta(s) (1-e^{is}) t}ds.
\end{align}
Recall the fact that $D_s^k\vp(0)=i^k\EE(Z^k)$ whenever the $k$th moments exists, where $D_s$ denotes differentiation with respect to $s$. Thus, assuming that $\vp$ is twice differentiable, elementary calculations\footnote{See Section \ref{sec: random setting} for more details in the random setting.} give that
$$
\EE(Z)=t
\qquad\text{ and }\qquad
\Var(Z)=t(1+2\Sg).
$$
 If no confusion arises, we will denote the underlying probability with $\mathbb{P},$ instead of $\mathbb{P}',$  and its moments with  $\mathbb{E},$ $\text{Var},$ etc., which are actually computed with the distribution $\nu_Z.$

\subsection{Further properties of the compound Poisson RV $Z$}
We now point out a few other properties enjoyed by compound Poisson random variables like our $Z.$ First of all such a variable may be written as \begin{equation}\label{eds}
Z:=\sum_{j=1}^N X_j,
\end{equation}
where the $X_j$ are  iid random variables  defined on  same probability space, and $N$ is Poisson distributed with parameter $\vta.$ There is a simple relationship between the characteristic function $\vp(s)$ of the variables $Z$ and $X_1$, namely we can write $\vp(s)=\phi_Z(s)$ where $\phi_Z(s)$ is given by 
\begin{equation}\label{ggg}
  \phi_Z(s)=e^{\vta (\phi_{X_1}(s)-1)};   
\end{equation}
see \cite{univ}. 
Then we use the well-known fact  that
\begin{equation}\label{cvb}
\mathbb{P}(Z=k)=\sum_{\ell=0}^{\infty}\mathbb{P}(N=\ell)\mathbb{P}(S_\ell=k),
\end{equation}
where $S_\ell=\sum_{i=1}^\ell X_i,$ and  $S_0$ is the random variable identically equal to $0.$ 
A standard result give $\mathbb{E}(Z)=\vta\mathbb{E}(X_1),$ and for $Z$ in (\ref{eds}) it follows immediately by taking the derivative of the characteristic function $\phi_Z$ in (\ref{ggg}) at $0,$ that $\mathbb{E}(Z)=t.$
Therefore $\vta=\frac{t}{\mathbb{E}(X_1)}.$
Moreover if follows from the definition of the random variable $Z_n$ (\ref{mil}) that (Gumbel's law): 
$$
\nu_Z(\{0\})=\lim_{n\rightarrow \infty} \mu(Z_n=0)=e
^{-\theta_0 t},$$ where $\theta_0$ is the extremal index.

Put now $k=0$ in (\ref{cvb}); then the only term which survives will be for $\ell=0$ 
and in this case we get $$\mathbb{P}(Z=0)=\mathbb{P}(N=0)=e^{-\vta}=e^{-\frac{t}{\mathbb{E}(X_1)}}$$
and therefore
$$
\mathbb{E}(X_1)=\theta_0^{-1}.
$$ This relationship was also obtained in \cite{HV20}, where a dynamical interpretation of the random variable $X_j$ was given. 
\\
Furthermore, using the fact that $\vta=t\ta_0$, the fact that $\vp(s)=\phi_Z(s)$, and \eqref{ggg}, we get that 
\begin{align}\label{eq det cf X}
    \phi_{X_1}(s)=\frac{\ta(s)(e^{is}-1)}{\ta_0}+1.   
\end{align}
Taking derivatives of $\phi_Z$ and $\phi_{X_1}$ we can additionally show\footnote{See Section \ref{sec: random setting} for details of the calculation in the random setting.}
\begin{align*}
    \Var(X_1) = \frac{\ta_0(1+2\Sg)-1}{\ta_0^2}.
\end{align*}
Note that the probability generating function (PGF) for our compound Poisson RV $Z$ is given by 
\begin{align}\label{eq pgf Z}
    G_Z(s)=e^{\vta(g_{X_1}(s)-1)},
\end{align}
where $g_{X_1}(s)$ is the PGF for the RV $X_1$. Set $\~g_{X_1}=\vta g_{X_1}$, and so we have $G_Z(s)=\exp(\~g_{X_1}(s)-\vta)$.
Then Hoppe's form of the generalized chain rule applied to $G_Z(s)$ gives 
\begin{align*}
    D_s^KG_Z(s) 
    =
    G_Z(s)\lt(\sum_{k=0}^K\frac{(-1)^k}{k!}\sum_{j=0}^k(-1)^j\binom{k}{j}(\~g_{X_1}(s))^{k-j}D_s^K(\~g_{X_1}^j)(s)\rt).
\end{align*}
Then 
\begin{align*}
    D_s^KG_Z(0) 
    &=
    G_Z(0)\lt(\sum_{k=0}^K\frac{(-1)^k}{k!}\sum_{j=0}^k(-1)^j\binom{k}{j}(\~g_{X_1}(0))^{k-j}D_s^K(\~g_{X_1}^j)(0)\rt)
    \\
    &=
    e^{-t\ta_0}\lt(\sum_{k=0}^K\frac{1}{k!}D_s^K(\~g_{X_1}^k)(0)\rt)
\end{align*}
since $G(0)=e^{-t\ta_0}$ and $g_{X_1}(0)=0$. Since $\PP(Z=K)=D_s^KG_Z(0)/K!$, we have 
\begin{align}\label{eq pgf prob Z}
    \PP(Z=K)
    &=
    \frac {e^{-t\ta_0}}{K!}\lt(\sum_{k=0}^K\frac{1}{k!}D_s^K(\~g_{X_1}^k)(0)\rt).
\end{align}

\begin{remark}\label{rem prob X}
Note that $\PP(Z=k)$ can be calculated using \eqref{eq P Zn} or \eqref{eq prob Z}.
To calculate $\PP(X_1=k)$ we can apply the L\'evy inversion formula to the CF $\phi_{X_1}$ to get 
\begin{align}\label{eq prob X inv det}
    \PP(X_1=k)
    =
    \lim_{T\rightarrow \infty}\frac{1}{2T}\int_{-T}^{T}e^{-isk}\lt(\frac{\ta(s)(e^{is}-1)}{\ta_0}+1\rt)ds.
\end{align}
Note that the inversion formula can be used to numerically approximate the distribution of $X_1$. 
Alternatively, we can recursively solve for $D_s^k\~g_{X_1}(0)$ (and thus $\PP(X_1=k)$) in terms of $\PP(Z=k)$ using \eqref{eq pgf prob Z}.
\end{remark}

\subsection{Deterministic Examples}
We now present several examples in the interval map setting we have just described.
We first show in Examples~\ref{ex det 1} and \ref{ex det 2} that if the \jaadd[r]{target}s $H_n$ are centered around a single point $x_0$, 
\jaadd[r]{whose orbit has nonempty intersection with the set of discontinuities for $T$, }
then the return time distribution one can get is either standard Poisson (if $x_0$ is aperiodic) or P\'olya-Aeppli (if $x_0$ is periodic). We then show in Example~\ref{ex det 3} that if the \jaadd[r]{target}s $H_n$ contain finitely many connected components which are centered around finitely many points, each of which has an orbit that is distinct from the others, then the return time distribution can be expressed as an independent sum of standard Poisson and P\'olya-Aeppli random variables. Following this  
classification of the return time distributions for interval maps with \jaadd[r]{target}s centered around finitely many points with non-overlapping orbits, we give two examples, Examples~\ref{ex det 4} and \ref{ex det 5}, describing the complicated behavior that occurs when the \jaadd[r]{target}s are centered around points \cgtadd[r]{on the same } orbit. These examples are related to clustering of extreme events created by multiple correlated maxima, investigated in \cite{AzevedoFFR_2016, CorreiaFF22}. 

\cgtadd[r]{We will say that $x_0\in I$ is \textit{simple} if it is a point of continuity for $T$ and if its forward orbit is disjoint from the discontinuity set of $T$. The following examples deal with targets centered at simple points. For instance, if $T$ is smooth -- not only piecewise smooth-- then every $x_0\in I$ is simple. For an analysis of extreme value laws associated to targets centered at non-simple points, we refer the reader to \cite{AFV15}.}

We begin with the simplest example of \gfadd[r]{targets } centered around an aperiodic point. 
\begin{example}\label{ex det 1}

\gfadd[r]{Suppose that $x_0\in H_n\sub I$, $n\ge 1$.}  
If $x_0$ is \cgtadd[r]{simple and } aperiodic, then using standard arguments with \eqref{cuc}, it is easy to show that \gfadd[r]{the $\beta_k(\ell)$ are zero for all $k\ge 0$, $0\le\ell\le k$, and thus, the $q_k(s)$ are zero for all $k\ge 0$ and $s\in\mathbb{R}$.}  
 In this case we have that $\ta\equiv 1$, and so, using (\ref{eq det scaling}),  the CF $\vp(s)=e^{-\theta (s) (1-e^{is}) t}$, which is the CF for a standard Poisson RV.  
\end{example}

\begin{example}\label{ex det 2}
\jaadd[r]{Now suppose that $x_0\in I$ is a simple periodic point of prime period $r\geq 1$, and that the \gfadd[r]{targets } $H_n$ are centered around the point $x_0$. Further suppose that the density $h$ and $DT^r$ are both continuous at $x_0$. }
Again using standard arguments with \eqref{cuc}, we see that all the $q_k=0$ for $k\not\equiv r-1 \pmod r$ and that for $k=br-1$, $b\geq 1$, we have 
$$
q_{br-1}(s)=\lim_{n\rightarrow \infty}\frac{(1-e^{is})}{\mu(H_n)}\int_{H_n\cap  T^{-br}(H_n)}e^{(b-1)is}h(x) dx=(1-e^{is})e^{(b-1)is}\al^b,
$$
since 
$$
\lim_{n\rightarrow \infty}\frac{\mu(H_n\cap  T^{-br}(H_n))}{\mu(H_n)}=\al^b,
$$
where $\al=1/|DT^r(x_0)|<1$. In this case $\theta$ becomes
\begin{align*}
    \theta(s)&=1-\sum_{b=1}^\infty q_{br-1}(s)
    =1-(1-e^{is})\sum_{b=1}^\infty e^{(b-1)is}\al^b
    =1-\frac{1-e^{is}}{\al^{-1}-e^{is}}
    =\frac{1-\al}{1-\al e^{is}}.
\end{align*}
Thus, using (\ref{eq det scaling}), the CF $\vp(s)$ is given by 
$$
\vp(s) = 
e^{-t(1-e^{is})\theta (s)  }
=
e^{-t(1-e^{is})(\frac{1-\al}{1-\al e^{is}})}
=
e^{-t(1-\al)(\frac{1-e^{is}}{1-\al e^{is}})},
$$
which is the CF of a P\'olya-Aeppli distributed random variable $Z$ with parameters $\rho=\al\in (0,1)$ and $\vta = t(1-\al)$. The probability mass distribution of $Z$ (see \cite[Section 9.6]{univ})is given by 
 \begin{equation*}
 \mathbb{P}(Z=k)=
 \begin{cases} 
    e^{-\vta} & \text{ for } k=0,
    \\
    e^{-\vta}\sum_{j=1}^k\frac{\vta^j}{j!} \rho^{k-j}(1-\rho)^j\binom{k-1}{j-1} & \text{ for } k\geq 1.
    \end{cases}
\end{equation*}
\end{example}

We now present several examples describing the return times distributions one can obtain in the deterministic setting with \gfadd[r]{targets } with multiple connected components. 

\begin{example}\label{ex det 3}
Suppose $x_1,\dots,x_M\in I$ \jaadd[r]{are simple and points of continuity for $h$, all of which} belong to different orbits, so that $T^k(x_i)\neq x_j$ for any $i,j,k\in\NN$ with $i\neq j$. Suppose that $H_n^{(j)}$ is an interval centered around the points $x_j$ 
such that for all $n$ sufficiently large we have $H_n^{(i)}\cap H_n^{(j)}=\emptyset$ 
for all $1\leq i\neq j\leq M$.  Set $H_n:=\cup_{j=1}^M H_n^{(j)}$ and we assume the following limit exists:
\begin{align*}
    p_j:=\lim_{n\to\infty}\frac{\mu(H_n^{(j)})}{\mu(H_n)}.
\end{align*}
Since we can write
\begin{align*}
    H_n\cap T^{-k}(H_n)
    =
    \bigcup_{j=1}^M \lt(H_n^{(j)}\cap T^{-k}(H_n^{(j)})\rt),
\end{align*}
we can decompose the integral 
\begin{align*}
    \int_{H_n\cap T^{-k}(H_n)} f(y) dy
    =
    \sum_{j=1}^M \int_{H_n^{(j)}\cap T^{-k}(H_n^{(j)})} f(y) dy.
\end{align*}
Thus, we denote
\begin{align*}
    I_{k,j}(s):=\lim_{n\to\infty}\frac{1-e^{is}}{\mu(H_n)}\int_{H_n^{(j)}\cap T^{-k}(H_n^{(j)})}e^{is S_{n,k}(y)}h(y)\, dy.
\end{align*}
which implies that $q_k(s)=\sum_{j=1}^M I_{k,j}(s).$
Suppose that $x_j$ is periodic with prime period $r_j$ \jaadd[r]{and that $DT^{r_j}$ is continuous at $x_j$} for $1\leq j\leq m$ (for $m\leq M$).  For each $1\leq j\leq m$ denote
$$
\al_j=1/|DT^{r_j}(x_j)|.
$$
For $m+1\leq j\leq M$ suppose that $x_j$ is aperiodic. Applying the same arguments from both examples above we obtain 
\begin{align*}
    I_{k,j}(s)=
    \begin{cases}
        (1-e^{is})\lt(e^{is\lt(b-1\rt)}p_j\al_j^b\rt)
        & \text{ if $x_j$ is periodic and }k=b r_j-1, \quad b\geq 1,
        \\
        0 & \text{ otherwise.}
    \end{cases}
\end{align*}
We set
\begin{align*}
    \kp_j(s) := \sum_{k=0}^\infty I_{k,j}(s),
\end{align*}
and thus, if $x_j$ is aperiodic ($m+1\leq j\leq M$), we have 
$\kp_j(s)\equiv 0$, and if $x_j$ is periodic ($1\leq j\leq m$), then following the calculation from the previous examples gives that 
\begin{align*}
    \kp_j(s)= \frac{p_j\al_j(1-e^{is})}{1-\al_je^{is}}.
\end{align*}
In order to calculate the characteristic function $\vp(s)$ for the random variable $Z$ coming from \eqref{as}, we first calculate $\ta(s)$. Collecting together the above calculations we can write 
\begin{align*}
    \ta(s)&=1-\sum_{k=0}^\infty q_k(s)
    =1-\sum_{j=1}^M \sum_{k=0}^\infty I_{k,j}(s)
    =1-\sum_{j=1}^M\kp_j(s)
    =1-\sum_{j=1}^m\frac{p_j\al_j(1-e^{is})}{1-\al_je^{is}}
    \\
    &=\frac{\prod_{k=1}^m(1-\al_ke^{is})
    -
    \sum_{j=1}^m(p_j\al_j(1-e^{is}))\prod_{k\neq j}(1-\al_ke^{is})
    }
    {\prod_{k=1}^m(1-\al_ke^{is})}
    \\
    &=
    \frac{\sum_{j=1}^Mp_j\prod_{k=1}^m(1-\al_ke^{is})
    -
    \sum_{j=1}^m(p_j\al_j(1-e^{is}))\prod_{k\neq j}(1-\al_ke^{is})
    }
    {\prod_{k=1}^m(1-\al_ke^{is})}
    \\
    &=\frac{\sum_{j=1}^m \lt(\lt(p_j(1-\al_je^{is})-p_j\al_j(1-e^{is})\rt)\prod_{k\neq j}(1-\al_ke^{is})\rt)
    +
    \sum_{j=m+1}^Mp_j\prod_{k=1}^m(1-\al_ke^{is}))
    }
    {\prod_{k=1}^m(1-\al_ke^{is})}
    \\
    &=\sum_{j=1}^m\frac{p_j(1-\al_j)}{1-\al_je^{is}}
    +\sum_{j=m+1}^Mp_j.
\end{align*}
Thus the CF $\vp(s)$ is given by 
\begin{align*}
    \vp(s)&=\exp\lt(-t(1-e^{is})\sum_{j=1}^m\frac{p_j(1-\al_j)}{1-\al_je^{is}}\rt)
    \cdot
    \exp\lt(-t(1-e^{is})\sum_{j=m+1}^Mp_j\rt)
    \\
    &=\exp\lt(-t(1-e^{is})\sum_{j=m+1}^Mp_j\rt)
    \cdot
    \prod_{j=1}^m
    \exp\lt(-tp_j(1-\al_j) \frac{1-e^{is}}{1-\al_je^{is}}\rt).
\end{align*}
Thus the random variable $Z$ can be written as a sum of independent random variables 
$Z=W_0+\sum_{j=1}^m W_j$ where $W_0$ is Poisson with parameter $\vta=t\sum_{j=m+1}^Mp_j$ and $W_j$ is P\'olya-Aeppli distributed with parameters $\rho_j=\al_j$ and $\vta_j=tp_j(1-\al_j)$ for each $1\leq j\leq m$.
 
\end{example}

We have thus proved the following theorem classifying the full extent of the return time distributions when the \gfadd[r]{targets } are centered around finitely many \jaadd[r]{simple } points with distinct orbits. 
\begin{theorem}
    For each $1\leq j\leq M$ and $n\in\NN$ suppose  $x_j\in H_n^{(j)}$ \jaadd[r]{is simple } 
    with $H_n=\cup_{j=1}^M H_n^{(j)}$. For $m\leq M$ and $1\leq j\leq m$ suppose the following 
    \begin{enumerate}
        \item $x_j$ is periodic with prime period $r_j$,
        \item \jaadd[r]{$x_j$ is a continuity point of $h$,}
        \item \jaadd[r]{$DT^{r_j}$ is continuous at $x_j$ and} $\al_j=\sfrac{1}{|DT^{r_j}(x_j)|}$,
        \item $H_n^{(j)}$ is centered around $x_j$ for each $n\in\NN$.
    \end{enumerate}   
    Now for $m+1\leq k\leq M$ suppose that $x_k$ is aperiodic. 
    Further suppose that $x_j\neq T^\ell(x_k)$ for any $1\leq j\neq k\leq M$ and any $\ell\in\NN$, and that the following limit exists
    \begin{align*}
        \lim_{n\to\infty}\frac{\mu(H_n^{(j)})}{\mu(H_n)}=p_j.
    \end{align*}
    Then the RV $Z$ defined by the CF $\vp(s)$ given in  \eqref{as} can be written as a sum of independent random variables
    $$
        Z=W_0+\sum_{j=1}^m W_j,
    $$
    where $W_0$ is standard Poisson with parameter $\vta_0=t\sum_{j=m+1}^M p_j$ and $W_j$ is P\'olya-Aeppli distributed with parameters $\rho_j=\al_j$ and $\vta_j=tp_j(1-\al_j)$ for each $1\leq j\leq m$.
\end{theorem}

We now give examples that show the complexity that arises when the \jaadd[r]{target}s $H_n$ are centered around multiple points with overlapping orbits.

\begin{example}\label{ex det 4}
    Suppose $x_1\in I$ is aperiodic. Let $x_2=T^b(x_1)$ for some $b\geq 1$. Suppose \jaadd[r]{that $x_1$ (and thus also $x_2$) is simple, and suppose} $H_n=H_n^{(1)}\cup H_n^{(2)}$ with $H_n^{(1)}\cap H_n^{(2)}=\emptyset$ and $H_n^{(j)}$ is centered around $x_j$ for $j=1,2$ and each $n\in\NN$. 
    Further suppose the following limit exists: 
\begin{align*}
 p_j:=\lim_{n\to\infty}\frac{\mu(H_n^{(j)})}{\mu(H_n)}
\end{align*}
for $j=1,2$, \jaadd[r]{and that $h$ and $DT^b$ are continuous at $x_1$.}  
Then all of the $q_k$ are equal to zero except for $q_{b-1}$ which is given by 
\begin{align}
    q_{b-1}(s)&=\lim_{n\to\infty}\frac{1-e^{is}}{\mu(H_n)}\int_{H_n\cap T^{-b}(H_n)} e^{is[\ind_{H_n}(T(y))+\dots +\ind_{H_n}(T^{b-1}(y))]}h(y)\, dy
    \nonumber\\
    &=\lim_{n\to\infty}\frac{1-e^{is}}{\mu(H_n)}\int_{H_n^{(1)}\cap T^{-b}(H_n^{(2)})} h(y)\, dy
    =(1-e^{is})\min\{p_1, p_2\al\},
    \label{eq det q}
\end{align}
where
$$
    \al=\frac{1}{|DT^b(x_1)|}.
$$
Setting $\Gm=\min\{p_1, p_2\al\}$,\footnote{Note that $\Gm=p_1$ when $H_n^{(1)}\sub T^{-b}(H_n^{(2)})$, and $\Gm=p_2\al$ otherwise.} then $\ta(s)=1-q_{b-1}(s)=1+(e^{is}-1)\Gm$, and so the CF $\vp(s)$ from \eqref{as} is given by 
\begin{align*}
    \vp(s)
    &=
    \exp\lt(t(e^{is}-1)(1+(e^{is}-1)\Gm)\rt).
\end{align*}
Note that $\vp$ is the CF for a RV $Z$ whose distribution is compound Poisson, which is neither standard Poisson nor P\'olya-Aeppli. In view of \eqref{ggg} and \eqref{eq det cf X}, we see that the CF of $X_1$ is given by 
\begin{align*}
    \phi_{X_1}(s)=1+\frac{(e^{is}-1)(1+\Gm(e^{is}-1))}{\Gm},
\end{align*}
where we have that $\ta_0=\Gm$ follows from similar calculations as \eqref{eq det q}; see \cite{AFGTV-TFPF}.

\end{example}

\begin{example}\label{ex det 5}

Suppose $x_1\in I$ is a periodic point with prime period $r>1$. Let $x_2=T^b(x_1)$ for some $1\leq b< r$. Suppose \jaadd[r]{that $x_1$ is simple, and consequently so is $x_2$. Further suppose that $x_1$ and $x_2$ are points of continuity for $h$, and suppose} $H_n=H_n^{(1)}\cup H_n^{(2)}$ with $H_n^{(1)}\cap H_n^{(2)}=\emptyset$ and $H_n^{(j)}$ is centered around $x_j$ for $j=1,2$ and each $n\in\NN$. 
Further suppose the following limit exists: 
\begin{align*}
    p_j:=\lim_{n\to\infty}\frac{\mu(H_n^{(j)})}{\mu(H_n)}
\end{align*}
for $j=1,2$ \jaadd[r]{and that $DT^r$ is continuous at $x_1.$ }
Set 
$$
\al:=\frac{1}{|DT^r(x_1)|}=\frac{1}{|DT^r(x_2)|},
$$
and set 
\begin{align*}
    \gm_1:=\frac{1}{|DT^{r-b}(x_2)|}
    \qquad\text{ and }\qquad
    \gm_2:=\frac{1}{|DT^{b}(x_1)|}. 
\end{align*}
As previously argued, we have 
\begin{align*}
    \lim_{n\to\infty}\frac{\mu\lt(H_n^{(j)}\cap T^{-k}(H_n^{(j)})\rt)}{\mu(H_n)}=p_j\al^{\frac{k}{r}}
\end{align*}
for $j=1,2$ and $k\equiv 0\pmod r$. 
Now if $k=ar+b$ then we have 
\begin{align*}
    \lim_{n\to\infty}\frac{\mu\lt(H_n^{(1)}\cap T^{-k}(H_n^{(2)})\rt)}{\mu(H_n)}
    &=
    \lim_{n\to\infty}\frac{\min\lt\{\mu(H_n^{(1)}), \mu(H_n^{(2)})\frac{1}{|DT^{ar+b}(x_1)|}\rt\}}{\mu(H_n)}
    \\
    &=\min\lt\{p_1, p_2\al^a\cdot\frac{1}{|DT^{b}(x_1)|}\rt\}=:\Gm_{1,2}(a)
\end{align*}
and similarly for $k=ar-b$ we have 
\begin{align*}
    \lim_{n\to\infty}\frac{\mu\lt(H_n^{(2)}\cap T^{-k}(H_n^{(1)})\rt)}{\mu(H_n)}
    &=
    \lim_{n\to\infty}\frac{\min\lt\{\mu(H_n^{(2)}), \mu(H_n^{(1)})\frac{1}{|DT^{ar-b}(x_2)|}\rt\}}{\mu(H_n)}
    \\
    &=\min\lt\{p_2 ,p_1\al^{a-1}\cdot\frac{1}{|DT^{r-b}(x_2)|}\rt\} =:\Gm_{2,1}(a).
\end{align*}
Note that there exist $a_{1,2}, a_{2,1}\geq 1$ such that $\Gm_{1,2}(a)=p_2\al^a\gm_2$ for all $a\geq a_{1,2}$ and $\Gm_{2,1}(a)=p_1\al^{a-1}\gm_1$ for all $a\geq a_{2,1}$.
For each $k,j,i\in\NN$ we set 
\begin{align*}
    I_{k,j,i}(s):=\lim_{n\to\infty}\frac{1-e^{is}}{\mu(H_n)}\int_{H_n^{(j)}\cap T^{-(k+1)}(H_n^{(i)})} e^{is[\ind_{H_n}(T(y))+\dots +\ind_{H_n}(T^{k}(y))]}h(y)\, dy
\end{align*}
and hence for each $k$ we have 
$$
    q_k(s):= I_{k,1,1}(s)+I_{k,2,2}(s)+I_{k,1,2}(s)+I_{k,2,1}(s).
$$
Following the previous arguments, for $k=ar-1$ we have that 
\begin{align*}
    I_{ar-1,j,j}(s)=(1-e^{is})(e^{is(a-1)}p_j\al^a)
\end{align*}
for $j=1,2$, and $I_{ar-1,1,2}(s)=I_{ar-1,2,1}(s)=0$, so 
\begin{align*}
    I_{ar-1}(s)
    &=I_{ar-1,1,1}(s)+I_{ar-1,2,2}(s)
    \\
    &=(1-e^{is})(e^{is(a-1)}(p_1\al^a+p_2\al^a))
    =(1-e^{is})(e^{is(a-1)}\al^a).
\end{align*}
Now for $k=ar+b-1$ we have $I_{ar+b-1,1,1}(s)=I_{ar+b-1,2,2}(s)=I_{ar+b-1,2,1}(s)=0$ and 
\begin{align*}
    I_{ar+b-1}(s) = I_{ar+b-1, 1,2}(s) = (1-e^{is})e^{is(a-1)}\Gm_{1,2}(a)
\end{align*}
and similarly for 
$k=ar-b-1$ we have $I_{ar-b-1,1,1}(s)=I_{ar-b-1,2,2}(s)=I_{ar-b-1,1,2}(s)=0$ and 
\begin{align*}
    I_{ar-b-1}(s) = I_{ar-b-1, 2,1}(s) = (1-e^{is})e^{is(a-1)}\Gm_{2,1}(a).
\end{align*}
Then we can write 
\begin{align}\label{q_k ex 5}
    \sum_{k=0}^\infty q_k(s)
    &=
    \sum_{a=1}^\infty I_{ar-b-1}(s)+I_{ar-1}(s)+I_{ar+b-1}(s).
\end{align}
We now have three cases.
\\

\textbf{Case 1:} 
If $\Gm_{1,2}(a)=p_2\al^a\gm_2$ and $\Gm_{2,1}(a)=p_1\al^{a-1}\gm_1$ for all $a\geq 1$\footnote{Note that this is the case if $p_1=p_2$.}, then using \eqref{q_k ex 5} we have 
\begin{align*}
    \sum_{k=0}^\infty q_k(s)
    &=
    (1-e^{is})\sum_{a=1}^\infty(e^{is(a-1)})(\al^a+p_2\gm_2\al^a+p_1\gm_1\al^{a-1})
    \\
    &=
    (1-e^{is})(p_1\gm_1+\al(p_2\gm_2+1))\sum_{a=1}^\infty(e^{is}\al)^{a-1}
    \\
    &=
    (p_1\gm_1+\al(p_2\gm_2+1))\frac{1-e^{is}}{1-\al e^{is}}.
\end{align*}

\textbf{Case 2:} Suppose that $a_{1,2}>1$ and $\Gm_{2,1}(a)=p_1\al^{a-1}\gm_1$ for all $a\geq 1$.  
Then 
\begin{align*}
    \Gm_{1,2}(a)=
    \begin{cases}
        p_1 & \text{ for all }a< a_{1,2}
        \\
        p_2\al^a\gm_2 & \text{ for all }a\geq a_{1,2}.
    \end{cases}
\end{align*}

Then using \eqref{q_k ex 5} we have
\begin{align*}
    \sum_{k=0}^\infty q_k(s)
    &=
    \sum_{a=1}^{a_{1,2}-1}I_{ar+b-1}(s) + \sum_{a=a_{1,2}}^\infty I_{ar+b-1}(s)+ \sum_{a=1}^\infty I_{ar-b-1}(s)+I_{ar-1}(s).
\end{align*}
The first two sums are given by 
\begin{align*}
    \sum_{a=1}^{a_{1,2}-1}I_{ar+b-1}(s) + \sum_{a=a_{1,2}}^\infty I_{ar+b-1}(s)
    &=
    \sum_{a=1}^{a_{1,2}-1}(1-e^{is})e^{is(a-1)}p_1 + \sum_{a=a_{1,2}}^\infty (1-e^{is})e^{is(a-1)}p_2\al^a\gm_2
    \\
    &=p_1(1-e^{is})\frac{e^{is(a_{1,2}-1)} -1}{e^{is}-1} + (1-e^{is})p_2\gm_2\al\frac{(e^{is}\al)^{a_{1,2}-1}}{1-\al e^{is}}
    \\
    &=p_1(1-e^{is(a_{1,2}-1)})+(1-e^{is})p_2\gm_2\al\frac{(e^{is}\al)^{a_{1,2}-1}}{1-\al e^{is}}, 
\end{align*}
and the third of the three sums is given by 
\begin{align*}
    \sum_{a=1}^\infty I_{ar-b-1}(s)+I_{ar-1}(s)
    &=
    \sum_{a=1}^\infty (1-e^{is})e^{is(a-1)}(\al^a+p_1\al^{a-1}\gm_1)
    \\
    &=(1-e^{is})\lt((\al+p_1\gm_1)\sum_{a=1}^\infty(e^{is}\al)^{a-1}\rt)
    \\
    &=(\al+p_1\gm_1)\frac{1-e^{is}}{1-\al e^{is}}.
\end{align*}
Combining these sums gives 
\begin{align*}
    \sum_{k=0}^\infty q_k(s) 
    =
    p_1(1-e^{is(a_{1,2}-1)})+(1-e^{is})p_2\gm_2\al\frac{(e^{is}\al)^{a_{1,2}-1}}{1-\al e^{is}}
    +
    (\al+p_1\gm_1)\frac{1-e^{is}}{1-\al e^{is}}.
\end{align*}

\textbf{Case 3:} Suppose that $\Gm_{1,2}(a)=p_2\al^a\gm_2$ for all $a\geq 1$ and $a_{2,1}>1$.  
Then 
\begin{align*}
    \Gm_{2,1}(a)=
    \begin{cases}
        p_2 & \text{ for all }a< a_{2,1}
        \\
        p_1\al^{a-1}\gm_1 & \text{ for all }a\geq a_{2,1}.
    \end{cases}
\end{align*}
Arguing similarly as in Case 2, we have 
\begin{align*}
    \sum_{k=0}^\infty q_k(s) 
    =
    p_2(1-e^{is(a_{2,1}-1)})+(1-e^{is})p_1\gm_1\al\frac{(e^{is}\al)^{a_{2,1}-1}}{1-\al e^{is}}
    +
    \al(1+p_2\gm_2)\frac{1-e^{is}}{1-\al e^{is}}.
\end{align*}

Note that we cannot have a fourth case where both $a_{1,2}>1$ and $a_{2,1}>1$.  
If this were to occur, then we would have that 
\begin{align*}
    \Gm_{1,2}(a)=
    \begin{cases}
        p_1 & \text{ for all }a< a_{1,2}
        \\
        p_2\al^a\gm_2 & \text{ for all }a\geq a_{1,2}
    \end{cases}
\end{align*}
and
\begin{align*}
    \Gm_{2,1}(a)=
    \begin{cases}
        p_2 & \text{ for all }a< a_{2,1}
        \\
        p_1\al^{a-1}\gm_1 & \text{ for all }a\geq a_{2,1}.
    \end{cases}
\end{align*}
In particular, the definition of $\Gm_{1,2}$ and $\Gm_{2,1}$ for $a=1 < a_{1,2},a_{2,1}$ imply that 
\begin{align*}
    p_1\leq p_2\al\gm_2
    \qquad\text{ and }\qquad 
    p_2\leq p_1\gm_1.
\end{align*}
Taken together, and noting that $\al,\gm_1,\gm_2\in(0,1)$, we arrive at the contradiction that 
\begin{align*}
    p_2\leq p_1\gm_1\leq p_2\al\gm_1\gm_2.
\end{align*}

In each of the four cases we then have the characteristic function $\vp(s)$ is given by
$$
    \vp(s)=\exp\lt(-t(1-e^{is})\lt(1-\sum_{k=0}^\infty q_k(s)\rt)\rt),
$$
which is clearly not the CF for a standard Poisson or P\'olya-Aeppli distribution. 

\end{example}

\section{The Random Setting: Main Results}\label{sec: random setting}
In this section we now move to the setting of random dynamical systems where we will present our main result. 
To begin, let $(\Om,\sF,m)$ be a probability space and $\sg:\Om\to\Om$ an ergodic, invertible map  which preserves the measure $m$, i.e.	
\begin{align*}
		m\circ\sg^{-1}=m.
	\end{align*}
	For each $\om\in \Om$, we take $\cJ_\om$ to be a closed subset of a complete metrizable space $M$ such that the map
	\begin{align*}
		\Om\ni \om\longmapsto \cJ_\om
	\end{align*}
	is a closed random set, i.e. $\cJ_\om\sub M$ is closed for each $\om\in\Om$ and the map $\om\mapsto \cJ_\om$ is measurable  
	(see \cite{crauel_random_2002}), and we consider the maps
	\begin{align*}
		T_\om:\cJ_\om\to \cJ_{\sg\om}.
	\end{align*}
	By $T_\om^n:\cJ_\om\to \cJ_{\sg^n\om}$ we mean the $n$-fold composition 
	\begin{align*}
		T_\om^n:=T_{\sg^n\om}\circ\dots\circ T_\om:\cJ_\om\to \cJ_{\sg^n\om}.
	\end{align*}
	Given a set $A\sub \cJ_{\sg^n\om}$ we let
	\begin{align*}
		T_\om^{-n}(A):=\set{x\in \cJ_\om:T_\om^n(x)\in A}
	\end{align*}
	denote the inverse image of $A$ under the map $T_\om^n$ for each $\om\in\Om$ and $n\geq 1$.
	Now let
	\begin{align*}
		\cJ:=\bigcup_{\om\in\Om}\set{\om}\times \cJ_\om\sub \Om\times M,
	\end{align*}
	and define the induced skew-product map $T:\cJ\to \cJ$ by
	\begin{align*}
		T(\om,x)=(\sg\om,T_\om(x)).
	\end{align*}
Let $\sB$ denote the Borel $\sg$-algebra of $M$ and let $\sF\otimes\sB$ be the product $\sg$-algebra on $\Om\times M$. Throughout the text we denote Lebesgue measure by $\Leb$. We suppose the following:
	
	\,
	
	\begin{enumerate}[align=left,leftmargin=*,labelsep=\parindent]
		\item[\mylabel{M1}{M1}] The map $T:\cJ\to \cJ$ is measurable with respect to $\sF\otimes\sB$.
	\end{enumerate} 
	
	\,

	\begin{definition}\label{def: random prob measures}	
		A measure $\mu$ on $\Om\times M$ with respect to the product $\sg$-algebra $\sF\otimes\sB$ is said to be \textit{random measure} relative to $m$ if it has marginal $m$, i.e. if
		$$
		\mu\circ\pi^{-1}_1=m.
		$$ 
		The disintegrations $\set{\mu_\om}_{\om\in\Om}$ of $\mu$ with respect to the partition $\lt(\{\om\}\times M\rt)_{\om\in\Om}$ satisfy the following properties:
		\begin{enumerate}
			\item For every $B\in\sB$, the map $\Om\ni\om\longmapsto\mu_\om(B)\in [0,\infty]$ is measurable, 
			\item For $m$-a.e. $\om\in\Om$, the map $\sB\ni B\longmapsto\mu_\om(B)\in [0,\infty]$ is a Borel measure.
		\end{enumerate}
		We say that the random measure $\mu=\set{\mu_\om}_{\om\in\Om}$ is a \textit{random probability measure} if for $m$-a.e. $\om\in\Om$ the fiber measure $\mu_\om$ is a probability measure. Given a set $Z=\cup_{\om\in\Om}\set{\om}\times Z_\om\sub\Om\times M$, we say that the random measure $\mu=\set{\mu_\om}_{\om\in\Om}$ is supported in $Z$ if $\supp(\mu)\sub Z$ and consequently $\supp(\mu_\om)\sub Z_\om$ for $m$-a.e. $\om\in\Om$. We let $\cP_\Om(Z)$ denote the set of all random probability measures supported in $Z$. We will frequently denote a random measure $\mu$ by $\set{\mu_\om}_{\om\in\Om}$.
	\end{definition}
	The following proposition from Crauel \cite{crauel_random_2002}, shows that a random probability measure $\set{\mu_\om}_{\om\in\Om}$ on $\cJ$ uniquely identifies a probability measure on $\cJ$.
	\begin{proposition}[\cite{crauel_random_2002}, Propositions 3.3]\label{prop: random measure equiv}
		If $\set{\mu_\om}_{\om\in\Om}\in\cP_\Om(\cJ)$ is a random probability measure on $\cJ$, then for every bounded measurable function $f:\cJ\to\RR$, the function 
		$$
		\Om\ni\om\longmapsto \int_{\cJ_\om} f(\om,x) \, d\mu_\om(x)
		$$ 
		is measurable and 
		$$
		\sF\otimes\sB\ni A\longmapsto\int_\Om \int_{\cJ_\om} \ind_A(\om,x) \, d\mu_\om(x)\, dm(\om)
		$$
		defines a probability measure on $\cJ$.
	\end{proposition}

	For $f:\cJ\to\RR$ we let 
	\begin{align*}
		S_n(f_\om):=\sum_{j=0}^{n-1}f_{\sg^j(\om)}\circ T_\om^j
	\end{align*}
	denote the Birkhoff sum of $f$ with respect to $T$. 
	We will consider a potential of the form $\vp_0:\cJ\to\RR$, and for each $n\geq 1$ we consider the weight $g_0^{(n)}:\cJ\to\RR$ whose disintegrations are given by
	\begin{align}\label{def: formula for g_om^n}
		g_{\om,0}^{(n)}:=\exp(S_n(\vp_{\om,0})) =\prod_{j=0}^{n-1}g_{\sg^j\om,0}^{(1)}\circ T_\om^j
	\end{align}
	for each $\om\in\Om$. We will often denote $g_{\om,0}^{(1)}$ by $g_{\om,0}$. We assume there exists a family of Banach spaces $\set{\cB_\om,\norm{\spot}_{\cB_\om}}_{\om\in\Om}$ of complex-valued functions on each $\cJ_\om$ with $g_{\om,0}\in\cB_\om$ such that the fiberwise (Perron-Frobenius) transfer operator $\cL_{\om,0}:\cB_\om\to\cB_{\sg\om}$ given by
	\begin{align}\label{glob def: closed tr op}
		\cL_{\om,0}(f)(x):=\sum_{y\in T_\om^{-1}(x)}f(y)g_{\om,0}(y), \quad f\in\cB_\om, \, x\in \cJ_{\sg\om}
	\end{align}
	is well defined.
	Using induction we see that iterates $\cL_{\om,0}^n:\cB_{\om}\to\cB_{\sg^n\om}$ of the transfer operator are given by
	\begin{align*}
		\cL_{\om,0}^n(f)(x):=\sum_{y\in T_\om^{-n}(x)}f(y)g_{\om,0}^{(n)}(y), \quad f\in\cB_\om, \, x\in \cJ_{\sg^n\om}.
	\end{align*}
	We let $\cB$ denote the space of functions $f:\cJ\to\RR$ such that $f_\om\in\cB_\om$ for each $\om\in\Om$ and we define the global transfer operator $\cL_0:
	\cB\to \cB$ by 
	$$
	    (\cL_0 f)_\om(x):=\cL_{\sg^{-1}\om,0}f_{\sg^{-1}\om}(x)
	$$ 
	for $f\in\cB$ and $x\in \cJ_\om$.
	We assume the following measurability assumption:
	
	\,
	
	\begin{enumerate}[align=left,leftmargin=*,labelsep=\parindent]
	
		\item[\mylabel{M2}{M2}] For every measurable function $f \in \cB$, the map 
		$(\om, x) \mapsto (\cL_0 f)_\om(x)$ is measurable.
		
	\end{enumerate}
	
	\,
	
	We suppose the following condition on the existence of a closed conformal measure. 

	\,
	
	\begin{enumerate}[align=left,leftmargin=*,labelsep=\parindent]
		\item[\mylabel{CCM}{CCM}] There exists a random probability  measure $\nu_0=\set{\nu_{\om,0}}_{\om\in\Om}\in \cP_\Om(\cJ)$ and measurable functions $\lm_0:\Om \to\RR\bs\set{0}$ and $\phi_0:\cJ\to (0,\infty)$ with $\phi_0\in\cB$ such that 
		\begin{align*}
			\cL_{\om,0}(\phi_{\om,0})=\lm_{\om,0}\phi_{\sg\om,0}
			\quad\text{ and }\quad
			\nu_{\sg\om,0}(\cL_{\om,0}(f))=\lm_{\om,0}\nu_{\om,0}(f)
		\end{align*}
		for all $f\in\cB_\om$ where $\phi_{\om,0}(\spot):=\phi_0(\om,\spot)$. Furthermore, we suppose that the fiber measures $\nu_{\om,0}$ are non-atomic and that $\lm_{\om,0}:=\nu_{\sg\om,0}(\cL_{\om,0}\ind)$ with $\log\lm_{\om,0}\in L^1(m)$. 
		We then define the random probability measure
		$\mu_0$ on $\cJ$ by
		\begin{align}\label{eq: def of mu_om,0}
			\mu_{\om,0}(f):=\int_{\cJ_\om} f\phi_{\om,0} \ d\nu_{\om,0},  \qquad f\in L^1(\nu_{\om,0}).
		\end{align}
	\end{enumerate}
	
	\,
	
	From the definition, one can easily show that $\mu_0$ is $T$-invariant, that is,  
	\begin{align}\label{eq: mu_om,0 T invar}
		\int_{\cJ_\om} f\circ T_\om \ d\mu_{\om,0}
		=
		\int_{\cJ_{\sg\om}} f \ d\mu_{\sg\om,0}, \qquad f\in L^1(\mu_{\sg\om,0}).
	\end{align}
	\begin{remark}
		Our Assumption \eqref{CCM} has been shown to hold in several random settings: random interval maps \cite{AFGTV20,AFGTV-IVC,AFGTV-TFPF}, random subshifts \cite{Bogenschutz_RuelleTransferOperator_1995a,mayer_countable_2015}, random distance expanding maps \cite{mayer_distance_2011}, random polynomial systems \cite{Bruck_Generalizediteration_2003}, and random transcendental maps \cite{mayer_random_2018}.
	\end{remark}

\subsection{Random Perturbations}
For each $n\in\NN$ we let $H_n\sub \cJ$ be measurable with respect to the product $\sg$-algebra $\sF\otimes\sB$ on $\cJ$ such that 
	\begin{enumerate}[align=left,leftmargin=*,labelsep=\parindent]
		
		\item[\mylabel{A}{A}]
		$H_n'\sub H_n$ for each $n'\leq n$.
	\end{enumerate}
	
	\,
	
	Then the sets $H_{\om,n}\sub \cJ_\om$ are uniquely determined by the condition that 
	\begin{align*}
		\set{\om}\times H_{\om,n}=H_n\cap\lt(\set{\om}\times \cJ_\om\rt),
	\end{align*}
	or equivalently that 
	\begin{align*}
		H_{\om,n}=\pi_2(H_n\cap(\set{\om}\times \cJ_\om)),
	\end{align*}
	where $\pi_2:\cJ\to \cJ_\om$ is the projection onto the second coordinate. The sets $H_{\om,n}$ are then $\nu_{\om,0}$-measurable, and \eqref{A} implies that 
	\begin{enumerate}[align=left,leftmargin=*,labelsep=\parindent]
		\item[\mylabel{A'}{A'}]$H_{\om,n'}\sub H_{\om,n}$ for each $n'\leq n$ and each $\om\in\Om$.
	\end{enumerate}

 For each $\oio$, each $s\in\RR$, and each $\nin$ we define the perturbed operator $\cL_{\om,n,s}:\cB_\om\to\cB_{\sg\om}$
 \begin{align*}
     \cL_{\om,n,s}(f):=\cL_{\om,0}(f\cdot e^{is\ind_{H_{\om,n}}}).
 \end{align*}
Note that if $s=0$ then $\cL_{\om,n,0}=\cL_{\om,0}$ for each $n\geq 1$. 
If we denote
$$
S_{\om,n,k}(x):=\sum_{j=0}^{k-1}{\ind}_{H_{\sg^j\om,n}}(T_\om^jx),
$$
then we have that
\begin{align}\label{eq pert iterates}
\cL_{\om,n,s}^k(f)=\cL_{\om,0}^n(e^{isS_{\om,n,k}}f).
\end{align}
Additionally we make the following assumption on our Banach spaces $\cB_\om$
	\begin{enumerate}[align=left,leftmargin=*,labelsep=\parindent]
		\item[\mylabel{B}{B}]
		$\ind_{H_{\om,n}}\in \cB_\om$ for each $n$ and each $\oio$, and there exists $C>0$ such that for $m$-a.e. $\oio$ we have $\|\ind_{H_{\om,n}}\|_{\cB_\om}\leq C$.
	\end{enumerate}
With a view toward applying Theorem 2.1.2 in \cite{AFGTV-TFPF} we calculate the following quantities

\begin{align}
	\Dl_{\om,n}(s)&:=\nu_{\sg\om,0}\left((\cL_{\om,0}-\cL_{\om,n,s})(\phi_{\om,0})\right)
	=\nu_{\sg\om,0}\left(\cL_{\om,0}(\phi_{\om,0}(1-e^{is\ind_{H_{\om,n}}}))\right)
	\nonumber\\
	&=\lm_{\om,0}\cdot\nu_{\om,0}(\phi_{\om,0}(1-e^{is\ind_{H_{\om,n}}}))
	=\lm_{\om,0}\cdot\mu_{\om,0}(1-e^{is\ind_{H_{\om,n}}})
 \nonumber\\
    &=\lm_{\om,0}(1-e^{is})\mu_{\om,0}(H_{\om,n}),
	\label{eq: Dl = mu H}
\end{align}
and
\begin{align}
	\eta_{\om,n}(s):&=\norm{\nu_{\sg\om,0}\left(\cL_{\om,0}-\cL_{\om,n,s}\right)}_{\cB_\om}
	=\sup_{\norm{\psi}_{\cB_\om}\leq 1}\nu_{\sg\om,0}\left(\cL_{\om,0}(\psi(1-e^{is\ind_{H_{\om,n}}}))\right)
	\nonumber
	\\
	&=\lm_{\om,0}\cdot\sup_{\norm{\psi}_{\cB_\om}\leq 1}\nu_{\om,0}\left(\psi(1-e^{is\ind_{H_{\om,n}}})\right).
	\label{etaineq}
\end{align}
\begin{remark}\label{rem check C6}
    If there exists some measurable, finite $m$-a.e. function $K:\Om\to(0,\infty]$ such that $\|\spot\|_{L^1(\nu_{\om,0})}\leq K_\om\|\spot\|_{\cB_\om}$, then the final equality of \eqref{etaineq} can be bounded by the following:
    \begin{align}
        \eta_{\om,n}(s)
        \leq
        K_\om\lm_{\om,0}|1-e^{is}|\nu_{\om,0}(H_{\om,n}).
        \label{etaineq2}
    \end{align}
    Note that the right hand side of \eqref{etaineq2} approaches zero as $n\to\infty$ if $\nu_{\om,0}(H_{\om,n})\to 0$ as $n\to\infty$. Furthermore, if $\cB_\om$ is the space of bounded variation functions, we may take $K_\om\equiv 2.$
\end{remark}

For each $\oio$,  $\nin$,  $s\in\RR$, and  $k\in\NN$, we define the quantities $q_{\om,n}^{(k)}(s)$ by the following formula:
\begin{align*}
	&q_{\om,n}^{(k)}(s)
	:=\frac
	{\nu_{\sg\om,0}\left((\cL_{\om,0}-\cL_{\om,n,s})(\cL_{\sg^{-k}\om,n,s}^k)(\cL_{\sg^{-(k+1)}\om,0}-\cL_{\sg^{-(k+1)}\om,n,s})(\phi_{\sg^{-(k+1)}\om,0})\right)}
	{\Dl_{\om,n}(s)}
	\\
	&=\frac{\lm_{\sg^{-(k+1)}\om,0}^{k+2}\cdot\mu_{\sg^{-(k+1)}\om,0}
        \left(
        \left(1-e^{is\ind_{H_{\om,n}}\circ T_{\sg^{-(k+1)}\om}^{k+1} }\right)
        \left(e^{is S_{\sg^{-k}\om,n,k}\circ T_{\sg^{-(k+1)}\om}}\right)
        \left(1-e^{is\ind_{H_{\sg^{-(k+1)}\om,n}}}\right)
		\right)
	}
	{\lm_{\om,0}(1-e^{is})\mu_{\om,0}(H_{\om,n})}
	\\
	&=\frac{\lm_{\sg^{-(k+1)}\om,0}^{k+1}\cdot\mu_{\sg^{-(k+1)}\om,0}
        \left(
\left(1-e^{is\ind_{H_{\om,n}}\circ T_{\sg^{-(k+1)}\om}^{k+1} }\right)
        \left(e^{is S_{\sg^{-k}\om,n,k}\circ T_{\sg^{-(k+1)}\om}}\right)
        \left(1-e^{is\ind_{H_{\sg^{-(k+1)}\om,n}}}\right)
		\right)
	}
	{(1-e^{is})\mu_{\om,0}(H_{\om,n})}
 \\
 &=\frac{\lm_{\sg^{-(k+1)}\om,0}^{k+1}(1-e^{is})^2\mu_{\sg^{-(k+1)}\om,0}
        \left(
\left(\ind_{H_{\om,n}}\circ T_{\sg^{-(k+1)}\om}^{k+1} \right)
        \left(e^{is S_{\sg^{-k}\om,n,k}\circ T_{\sg^{-(k+1)}\om}}\right)
        \left(\ind_{H_{\sg^{-(k+1)}\om,n}}\right)
		\right)
	}
	{(1-e^{is})\mu_{\om,0}(H_{\om,n})}
  \\
 &=\frac{\lm_{\sg^{-(k+1)}\om,0}^{k+1}(1-e^{is})\mu_{\sg^{-(k+1)}\om,0}
        \left(
\left(\ind_{T_{\sg^{-(k+1)}\om}^{-(k+1)}(H_{\om,n})} \right)
        \left(e^{is S_{\sg^{-k}\om,n,k}\circ T_{\sg^{-(k+1)}\om}}\right)
        \left(\ind_{H_{\sg^{-(k+1)}\om,n}}\right)
		\right)
	}
	{\mu_{\om,0}(H_{\om,n})}.
\end{align*}
For notational convenience we define the quantity $\hat q_{\om,n}^{(k)}(s)$ by 
\begin{align}\label{def of hat q}
	\hat q_{\om,n}^{(k)}(s):&=
	\frac{(1-e^{is})\mu_{\sg^{-(k+1)}\om,0}
        \left(
\left(\ind_{T_{\sg^{-(k+1)}\om}^{-(k+1)}(H_{\om,n})} \right)
        \left(e^{is S_{\sg^{-k}\om,n,k}\circ T_{\sg^{-(k+1)}\om}}\right)
        \left(\ind_{H_{\sg^{-(k+1)}\om,n}}\right)
		\right)
	}
	{\mu_{\om,0}(H_{\om,n})},
\end{align}
and thus we have that 
\begin{align*}
	\hat q_{\om,n}^{(k)}(s)=\left(\lm_{\sg^{-(k+1)}\om,0}^{k+1}\right)^{-1} q_{\om,n}^{(k)}(s).
\end{align*}

We assume the following conditions hold for each $\nin$ sufficiently large and each $s\in\RR$:
\begin{enumerate}[align=left,leftmargin=*,labelsep=\parindent]
	\item[\mylabel{C1}{C1}]

There exists $C_1\geq 1$ such that for $m$-e.a. $\om\in\Om$ we have 
	\begin{align*}
		C_1^{-1}\leq \cL_{\om,0}\ind\leq C_1.
	\end{align*}
 
	\item[\mylabel{C2}{C2}]  There exist 
    $\nu_{\om,n,s}\in\cB_\om^*$ (the dual of $\cB_\om$)
	and measurable functions $\lm_{n,s}:\Om\to\CC\bs\{0\}$ 
	with $\log|\lm_{\om,n,s}|\in L^1(m)$ and $\phi_{n,s}:\cJ\to \CC$ such that 
	\begin{align*}
		\cL_{\om,n,s}(\phi_{\om,n,s})=\lm_{\om,n,s}\phi_{\sg\om,n,s}
		\quad\text{ and }\quad
		\nu_{\sg\om,n,s}(\cL_{\om,n,s}(f))=\lm_{\om,n,s}\nu_{\om,n,s}(f)
	\end{align*}
	for all $f\in\cB_\om$. Furthermore we assume that for $m$-a.e. $\om\in\Om$
	$$
	\nu_{\om,0}(\phi_{\om,n,s})=1
	\quad\text{ and } \quad
	\nu_{\om,0}(\ind)=1.
	$$
	\item[\mylabel{C3}{C3}] There exist operators $Q_{\om,0}, Q_{\om,n,s}:\cB_\om\to\cB_{\sg\om}$ such that for $m$-a.e. $\om\in\Om$ and each $f\in\cB_\om$ we have
    \begin{align*}
		\lm_{\om,0}^{-1}\cL_{\om,0}(f)=\nu_{\om,0}(f)\cdot\phi_{\sg\om,0}+Q_{\om,0}(f)
	\end{align*}
    and 
	\begin{align*}
		\lm_{\om,n,s}^{-1}\cL_{\om,n,s}(f)=\nu_{\om,n,s}(f)\cdot\phi_{\sg\om,n,s}+Q_{\om,n,s}(f).
	\end{align*}
	Furthermore, for $m$-a.e. $\om\in\Om$ we have
    \begin{align*}
		Q_{\om,0}(\phi_{\om,0})=0
		\quad\text{ and }\quad
		\nu_{\sg\om,0}(Q_{\om,0}(f))=0
	\end{align*}
    and 
	\begin{align*}
		Q_{\om,n,s}(\phi_{\om,n,s})=0
		\quad\text{ and }\quad
		\nu_{\sg\om,n,s}(Q_{\om,n,s}(f))=0.
	\end{align*}
	\item[\mylabel{C4}{C4}] 
	There exists $C>0$ and a summable sequence $\al(N)>0$ (independent of $\om$) with $\al:=\sum_{N=1}^\infty\al(N)<\infty$ such that for each $f\in\cB$, for $m$-a.e. $\om\in\Om$, and each $N\in\NN$
    \begin{align*}
		\norm{Q_{\om,0}^N f_{\om}}_{\infty,\sg^N\om}\leq C\al(N)\norm{f_{\om}}_{\cB_{\om}}.
	\end{align*}
    and
	\begin{align*}
		\sup_{n> 0}\norm{Q_{\om,n,s}^N f_{\om}}_{\infty,\sg^N\om}\leq C\al(N)\norm{f_{\om}}_{\cB_{\om}}.
	\end{align*}

	\item[\mylabel{C5}{C5}] 

 There exists $C_2\geq 1$ such that
	\begin{align*}
		\norm{\phi_{\om,0}}_{\cB_\om}\leq C_2
		\quad\text{ and }\quad
		\sup_{n>0}\norm{\phi_{\om,n,s}}_{\infty,\om}\leq C_2
	\end{align*}
	for $m$-a.e. $\om\in\Om$.

	\item[\mylabel{C6}{C6}] 
	For $m$-a.e. $\om\in\Om$ we have
	\begin{align*}
		\lim_{n\to \infty}\eta_{\om,n}=0. 
	\end{align*}
	
	\item[\mylabel{C7}{C7}]
 	There exists $C_3\geq 1$ such that for all $n>0$ sufficiently large we have
	\begin{align*}
		\essinf_\om\inf\phi_{\om,0}\geq C_3^{-1} >0.
	\end{align*}
	
	\item[\mylabel{C8}{C8}] 
	For $m$-a.e. $\om\in\Om$ 
	we have that the limit $\hat{q}_{\om,0}^{(k)}(s):=\lim_{n\to 0} \hat{q}_{\om,n,s}^{(k)}$ exists for each $k\geq 0$, where $\hat{q}_{\om,n,s}^{(k)}$ is as in \eqref{def of hat q}.

\item[\mylabel{S}{S}]
For any fixed random scaling function $t\in L^\infty(m)$ with $t>0$, we may find a sequence of functions $\xi_n\in L^\infty(m)$  and a constant $W<\infty$ satisfying 
	$$
        \mu_{\omega,0}(H_{\om,n})=(t_\omega+\xi_{\omega,n})/n, \mbox{ for a.e.\ $\omega$ and each $n\ge 1$},
	$$
where:\\
(i)
 $\lim_{n\to\infty} \xi_{\omega,n}=0$ for a.e.\ $\omega$ and \\
 (ii) $|\xi_{\omega,n}|\le W$ for a.e.\ $\omega$ and all $n\ge 1$.

\end{enumerate}
\begin{remark}
    Assumption \eqref{C4} can easily be shown to imply the following decay of correlations condition: 
    \begin{itemize}
        \item For $m$-a.e., every $N\in\NN$, 
        every $f\in L^1(\mu)$ and every $h\in\cB$ we have 
        \begin{align}\label{dec}
            \lt|
			\mu_{\om,0}
			\lt(\lt(f_{\sg^{N}(\om)}\circ T_{\om}^N\rt)h_{\om} \rt)
			-
			\mu_{\sg^{N}(\om),0}(f_{\sg^{N}(\om)})\mu_{\om,0}(h_{\om})
		      \rt|
            \leq C\|f_{\sg^N\om}\|_{L^1(\mu_{\sg^N\om,0})}\|h_\om\|_{\cB_\om}\al(N).
        \end{align}
    \end{itemize}
\end{remark}

\begin{remark}({\em On the scaling \eqref{S} for non-random $t$})\label{otsf}
If $t$ is non-random, there may be situations in which it is easier to check (\ref{S})(i) using natural annealed quantities. 
A scaling that has been proposed in the quenched random setting reads (in our notation) for $n\ge 1:$ 
\begin{equation}\label{nnss}
n\int \mu_{\omega,0}(H_{\omega,n})dm(\omega)=t
\end{equation}
see \cite{RSV014} \svadd[r]{
and \cite{AHV24}}. 

We claim that the scaling (\ref{nnss}) together with an additional {\em annealed} assumption: for any $t>0,$ there exist $\kappa(t)>0$ and  a bounded constant $C(t)$ such that
\begin{equation}\label{2Y}
\int |n\mu_{\omega,0}(H_{\omega,n})-t|dm(\omega)\le \frac{C(t)}{n^{1+\kappa(t)}}.
\end{equation}
imply condition \eqref{S}(i). 
It is sufficient to show, by Borel-Cantelli, that for any $\epsilon>0, \sum_n m(|n\mu_{\omega,0}(H_{\omega,n})-t|>\epsilon)<\infty$. By Tchebychev's inequality we  have 
\begin{eqnarray*}
m(|n\mu_{\omega,0}(H_{\omega,n})-t|>\epsilon)&\le& \frac{1}{\epsilon}\mathbb{E}(|n\mu_{\omega,0}(H_{\omega,n})-t|).
\end{eqnarray*}
Thus (\ref{nnss}) and (\ref{2Y}) imply part (i) of our condition \eqref{S}: 
$$
n  \mu_{\omega,0}(H_{\omega,n})\rightarrow t, \ a.s.
$$
To ensure that part (ii) of \eqref{S} holds, 
it suffices to assume that there exists $K>0$ such that 
\begin{align*}
    \limsup_{n\to\infty}n\mu_{\om,0}(H_{\om,n})\leq K+t.
\end{align*}
Indeed, this immediately implies that 
\begin{align*}
    -K\leq n\mu_{\om,0}(H_{\om,n})-t\leq K
\end{align*}
for all $n$ sufficiently large.

In the other direction, it is immediate to see that condition \eqref{S} (with non-random $t$) implies a limit version of (\ref{nnss}), namely $n\int \mu_{\omega,0}(H_{\omega,n})dm(\omega)\rightarrow t$, as $n\rightarrow \infty.$

\end{remark}
\begin{definition}\label{def rand pert sys}
    We will call the collection $(\mathlist{\bcomma}{\Om, m, \sg, \cJ, T, \cB, \cL_0, \nu_0, \phi_0, H_n})$ a random perturbed system if the assumptions \eqref{M1}, \eqref{M2}, \eqref{CCM}, \eqref{A}, \eqref{B}, \eqref{C1}--\eqref{C7}
    are satisfied. 
\end{definition}
\begin{remark}
    See \cite{AFGTV-TFPF} for examples of systems which satisfy the assumptions of a Definition \ref{def rand pert sys} for perturbations similar to what we consider here. 
\end{remark}
Note that 
	\begin{align}
		\lm_{\om,0}-\lm_{\om,n,s}&=\lm_{\om,0}\nu_{\om,0}(\phi_{\om,n,s})-\nu_{\sg\om,0}(\lm_{\om,n,s}\phi_{\sg\om,n,s})\nonumber\\
		&=\nu_{\sg\om,0}(\cL_{\om,0}(\phi_{\om,n,s}))-\nu_{\sg\om,0}(\cL_{\om,n,s}(\phi_{\om,n,s}))\nonumber\\
		&=\nu_{\sg\om,0}\left((\cL_{\om,0}-\cL_{\om,n,s})(\phi_{\om,n,s})\right).\label{diff eigenvalues identity unif}
	\end{align}
	It then follows from \eqref{etaineq}, \eqref{etaineq2}, and \eqref{diff eigenvalues identity unif} that
    \begin{align}\label{conv eigenvalues unif}
		\absval{\lm_{\om,0}-\lm_{\om,n,s}}\leq C_2\eta_{\om,n}(s).
	\end{align}
Taking \eqref{conv eigenvalues unif} together with the assumption \eqref{C6}, we see that for each $s\in\RR$
\begin{align}\label{conv eigenvalues unif2}
    \lm_{\om,n,s}\to\lm_{\om,0}
\end{align}
as $n\to\infty.$
Using \eqref{CCM}, \eqref{eq pert iterates}, \eqref{C3}, and \eqref{C2}, for $\psi_\om\in\cB_\om$ we can write 
\begin{align*}
    \nu_{\om,0}\lt(\psi_\om e^{is S_{\om,n,k}}\rt)
    &=
    (\lm_{\om,0}^k)^{-1}\nu_{\sg^k\om,0}\lt(\cL_{\om,0}\lt(\psi_\om e^{is S_{\om,n,k}}\rt)\rt)
    =
    (\lm_{\om,0}^k)^{-1}\nu_{\sg^k\om,0}\lt(\cL_{\om,n,s}^k\psi_\om\rt)
    \\
    &=
    \frac{\lm_{\om,n,s}^k}{\lm_{\om,0}^k}
    \lt(\nu_{\om,n,s}(\psi_\om)\nu_{\sg^k\om,0}(\phi_{\sg^k\om,n,s})+\nu_{\sg^k\om,0}(Q_{\om,n,s}^k\psi_\om)\rt)
    \\
    &=
    \frac{\lm_{\om,n,s}^k}{\lm_{\om,0}^k}
    \lt(\nu_{\om,n,s}(\psi_\om)+\nu_{\sg^k\om,0}(Q_{\om,n,s}^k\psi_\om)\rt).
\end{align*}
In particular, if $\psi_\om=\ind$ we have 
\begin{align}\label{eq: nu0 of pert sum}
    \nu_{\om,0}\lt(e^{is S_{\om,n,k}}\rt)
    &=
    \frac{\lm_{\om,n,s}^k}{\lm_{\om,0}^k}
    \lt(\nu_{\om,n,s}(\ind)+\nu_{\sg^k\om,0}(Q_{\om,n,s}^k\ind)\rt),
\end{align}
and if $\psi_\om=\phi_{\om,0}$ we have 
\begin{align}\label{eq: mu0 of pert sum}
    \mu_{\om,0}\lt(e^{is S_{\om,n,k}}\rt)
    &=
    \frac{\lm_{\om,n,s}^k}{\lm_{\om,0}^k}
    \lt(\nu_{\om,n,s}(\phi_{\om,0})+\nu_{\sg^k\om,0}(Q_{\om,n,s}^k\phi_{\om,0})\rt).
\end{align}
In order to apply Theorem 2.1.2 in \cite{AFGTV-TFPF}, we need to check assumptions (P1)-(P9), however as it is clear that our assumptions \eqref{C1}--\eqref{C8} directly imply (P1)-(P6) as well as (P9), we have only to check the assumptions (P7) and (P8) of \cite{AFGTV-TFPF}. This is done in the following lemma. 
\begin{lemma}\label{lem: checking P7 and P8}
	Given a random perturbed system $(\mathlist{\bcomma}{\Om, m, \sg, \cJ, T, \cB, \cL_0, \nu_0, \phi_0, H_n})$, for $m$-a.e. $\oio$ and each $s\in\RR$ we have that 
	\begin{align}\label{eq: lem check P7}
		\lim_{n\to\infty}\nu_{\om,n,s}(\phi_{\om,0})=1,
	\end{align}
	and 
	\begin{align}\label{eq: lem check P8}
		\lim_{k\to\infty}\limsup_{n \to \infty}(\Dl_{\om,n}(s))^{-1} \nu_{\sg\om,0}\lt((\cL_{\om,0}-\cL_{\om,n,s})(Q_{\sg^{-k}\om,n,s}^k\phi_{\sg^{-k}\om,0})\rt)=0.
	\end{align}
    Thus, (P7) and (P8) of \cite{AFGTV-TFPF} follow from \eqref{eq: lem check P7} and \eqref{eq: lem check P8} respectively.
\end{lemma}
\begin{proof}
	First, using \eqref{eq: Dl = mu H}, we note that if $\mu_{\om,0}(H_{\om,n})>0$ then  $\Dl_{\om,n}(s)\neq 0$.
	To prove \eqref{eq: lem check P7}, we note that for fixed $k\in\NN$ we have 
	\begin{align}\label{eq: lim ratio times measure is 1}
		\lim_{n\to\infty}\frac{\lm_{\om,0}^k}{\lm_{\om,n,s}^k}\mu_{\om,0}\lt(e^{is S_{\om,n,k}}\rt)=1
	\end{align}
	since $\lm_{\om,0}^k/\lm_{\om,n,s}^k\to 1$ (by \eqref{conv eigenvalues unif2}) and since \eqref{C5} and \eqref{C6} together imply that $\mu_{\om,0}\lt(e^{is S_{\om,n,k}}\rt)\to 1$ as $n\to \infty$. 
	Using  \eqref{eq: mu0 of pert sum} we can write 
	\begin{align}\label{eq: solve for nu eps}
		\nu_{\om,n,s}(\phi_{\om,0})
		= \frac{\lm_{\om,0}^k}{\lm_{\om,n,s}^k}\mu_{\om,0}\lt(e^{is S_{\om,n,k}}\rt)
		- 
		\nu_{\sg^k\om,0}(Q_{\om,n,s}^k(\phi_{\om,0})),
	\end{align}
	and thus using \eqref{eq: lim ratio times measure is 1} and \eqref{C4}, 
	for each $\om\in\Om$ and each $k\in\NN$ we can write 
	\begin{align*}
		\lim_{n\to\infty}\absval{1-\nu_{\om,n,s}(\phi_{\om,0})}
		&\leq
		\lim_{n\to\infty}
		\absval{ 1-\frac{\lm_{\om,0}^k}{\lm_{\om,n,s}^k}\mu_{\om,0}\lt(e^{is S_{\om,n,k}}\rt)}
		+
		\norm{Q_{\om,n,s}^k(\phi_{\om,0})}_{\infty,\sg^k\om}
		\\
		&\leq
		C_{\phi_0}\al(k)\norm{\phi_{\om,0}}_{\cB_\om}.
	\end{align*}
	As this holds for each $k\in\NN$ and as the right-hand side of the previous equation goes to zero as $k\to\infty$, we must in fact have that  
	\begin{align*}
		\lim_{n\to\infty}\absval{1-\nu_{\om,n,s}(\phi_{\om,0})}=0,
	\end{align*}
	which yields the first claim.
	
	Now, for the second claim, using \eqref{eq: Dl = mu H}, we note that \eqref{C7} implies 
	\begin{align*}
		&\lt|(\Dl_{\om,n}(s))^{-1} \nu_{\sg\om,0}\lt((\cL_{\om,0}-\cL_{\om,n,s})(Q_{\sg^{-N}\om,\ep}^k\phi_{\sg^{-k}\om,0})\rt)\rt|
        \\
        &\quad=
		\lt|\frac{\nu_{\sg\om,0}\lt(\cL_{\om,0}\lt(\lt(1-e^{is\ind_{H_{\om,n}}}\rt)Q_{\sg^{-k}\om,n,s}^k(\phi_{\sg^{-k}\om,0}) \rt)\rt) }{\lm_{\om,0}(1-e^{is})\mu_{\om,0}(H_{\om,n})}\rt|
		=
		\lt|\frac{\nu_{\om,0}\lt(\lt(1-e^{is\ind_{H_{\om,n}}}\rt)Q_{\sg^{-k}\om,n,s}^k(\phi_{\sg^{-k}\om,0}) \rt) }{(1-e^{is})\mu_{\om,0}(H_{\om,n})}\rt|
		\\
		&\quad=
		\frac{\nu_{\om,0}(H_{\om,n})}{\mu_{\om,0}(H_{\om,n})}\norm{Q_{\sg^{-k}\om,n,s}^k(\phi_{\sg^{-k}\om,0})}_{\infty,\om}
		\leq C_3\norm{Q_{\sg^{-k}\om,n,s}^k(\phi_{\sg^{-k}\om,0})}_{\infty,\om}.
	\end{align*}
	Thus, letting $n\to\infty$ first and then $k\to\infty$, the second claim follows from \eqref{C4}.
\end{proof}

Now, for all $\ell\geq 0$, we define the following: 
\begin{align*}
    \bt_{\om,n}^{(k)}(\ell)
    :=
    \frac{\mu_{\sg^{-(k+1)}\om,0}
    \lt( \lt\{x\in H_{\sg^{-(k+1)}\om,n}: T_{\sg^{-(k+1)}\om}^{k+1}(x)\in H_{\om,n}, \sum_{j=1}^k\ind_{H_{\sg^{-(k+1)+j}\om,n}}(T_{\sg^{-(k+1)}\om}^j(x))=\ell\rt\}\rt)}{\mu_{\om,0}(H_{\om,n})}.
\end{align*}
Thus, in light of \eqref{def of hat q}, we have that 
\begin{align}\label{eq q with beta}
    \hat{q}_{\om,n}^{(k)}(s)
    =
    (1-e^{is})\sum_{\ell=0}^ke^{i\ell s}\bt_{\om,n}^{(k)}(\ell).
\end{align}
Note that 
\begin{align*}
    \sum_{\ell=0}^k\bt_{\om,n}^{(k)}(\ell)
    =
    \frac{\mu_{\sg^{-(k+1)}\om,0}\lt(H_{\sg^{-(k+1)}\om,n}\cap T_{\sg^{-(k+1)}\om}^{-(k+1)}(H_{\om,n})\rt)}{\mu_{\om,0}\lt( H_{\om,n}\rt)}.
\end{align*}
Using the decay of correlations condition \eqref{dec}, 
which follows from our assumption \eqref{C4}, together with the assumption \eqref{B} and the $T$-invariance of $\mu_0$, gives that 
\begin{align}\label{eq: beta bound}
    \sum_{\ell=0}^k\bt_{\om,n}^{(k)}(\ell)
    &\leq 
    \frac{\mu_{\om,0}(H_{\om,n})\mu_{\sg^{-(k+1)}\om,0}\lt(T_{\sg^{-(k+1)}\om}^{-(k+1)}(H_{\om,n})\rt)+C\mu_{\om,0}(H_{\om,n})\|\ind_{H_{\sg^{-(k+1)}\om,n}}\|_{\cB_{\sg^{-(k+1)}\om}}\al(k+1)}{\mu_{\om,0}(H_{\om,n})}
    \nonumber\\
    &=\mu_{\sg^{-(k+1)}\om,0}\lt(T_{\sg^{-(k+1)}\om}^{-(k+1)}(H_{\om,n})\rt)+C\|\ind_{H_{\sg^{-(k+1)}\om,n}}\|_{\cB_{\sg^{-(k+1)}\om}}\al(k+1)
    \nonumber\\
    &=\mu_{\om,0}(H_{\om,n})+C\|\ind_{H_{\sg^{-(k+1)}\om,n}}\|_{\cB_{\sg^{-(k+1)}\om}}\al(k+1)
    \nonumber\\
    &\leq \mu_{\om,0}(H_{\om,n})+C^2\al(k+1).
\end{align}
In view of \eqref{eq q with beta}, we now claim that for each $0\leq \ell\leq k$ the quantities $\bt_{\om,n}^{(k)}(\ell)$ converge as $n\to\infty$.
\begin{lemma}\label{lem beta limits}
    If \eqref{C8} holds, then the limit 
\begin{align*}
    \bt_{\om,0}^{(k)}(\ell)
    :=\lim_{n\to\infty}\bt_{\om,n}^{(k)}(\ell)    
\end{align*}
exists for each $0\leq \ell\leq k$. 
\end{lemma}
\begin{proof}
  Note that \eqref{eq q with beta} applied to $s=1, 2, \dots, k+1$,
yields, for each $n, k, \om$,
\begin{equation}\label{eq:QsAndBetas}
\begin{pmatrix}
\hat{q}_{\om,n}^{(k)}(1) \\
\hat{q}_{\om,n}^{(k)}(2) \\
\vdots \\
\hat{q}_{\om,n}^{(k)}(k+1) 
\end{pmatrix} =
D_k M_k
 \begin{pmatrix}
\bt_{\om,n}^{(k)}(0) \\
\bt_{\om,n}^{(k)}(1) \\
\vdots \\
\bt_{\om,n}^{(k)}(k)  
\end{pmatrix},
\end{equation}
where
$$D_k=\begin{pmatrix}
1-e^i & 0 & \dots & 0\\
0 & 1-e^{2i} & \dots & 0\\
\vdots\\
0 & 0 & \dots & 1-e^{(k+1)i}\\
\end{pmatrix} \quad \text{ and }\quad
M_k=\begin{pmatrix}
1 & e^i & \dots & e^{ki}\\
1 & e^{2i} & \dots & e^{2ki}\\
\vdots\\
1 & e^{(k+1)i} & \dots & e^{(k+1)ki}\\
\end{pmatrix}.$$
Clearly, $D_k$ is invertible. The next paragraph shows that $M_k$ is invertible. Thus, pre-multiplying \eqref{eq:QsAndBetas} by $M_k^{-1}D_k^{-1}$ yields an expression for the $\beta_n$'s as linear combinations of $\hat{q}_n$'s. Taking limits as $n\to \infty$, and recalling that \eqref{C8} ensures the limits exist for $\hat{q}_n$'s, the lemma follows.

To finish the proof, we show that $M_k$ is invertible. This is a consequence of (Baker's version of) Lindemann–Weierstrass theorem \cite[Theorem 1.4]{Baker-Transcendental}, stating that if 
    $a_1, ..., a_N$ are algebraic numbers, and $\alpha_1, \dots, \alpha_N$ are distinct algebraic numbers, then the equation
$$ a_{1}e^{\alpha _{1}}+a_{2}e^{\alpha _{2}}+\cdots +a_{N}e^{\alpha _{N}}=0$$
has only the trivial solution $a_i = 0$ for all $1\leq i \leq N$.
Indeed, $\det M_k = \sum_{j=0}^N a_j e^{ij}$, where 
$N=\sum_{\ell=0}^{k} \ell (\ell+1)$,
 $a_j$ is an integer for each $0\leq j \leq N$, and $a_N=1$. Thus, the Lindemann–Weierstrass theorem (with $\alpha_j = ij$) implies $\det M_k \neq0$ and thus $M_k$ is invertible.
\end{proof}
It follows from Lemma \ref{lem beta limits} that (assuming \eqref{C8} holds) we have
\begin{align*}
    \hat q_{\om,0}^{(k)}(s)
    =
    \lim_{n\to\infty}\hat q_{\om,n}^{(k)}(s)
    &=
    (1-e^{is})\sum_{\ell=0}^ke^{i\ell s}\bt_{\om,0}^{(k)}(\ell),
\end{align*}  
and furthermore, it follows from \eqref{eq: beta bound} that 
\begin{align}\label{eq: beta0 bound}
    \sum_{\ell=0}^k\bt_{\om,0}^{(k)}(\ell) = \lim_{n\to\infty} \sum_{\ell=0}^k\bt_{\om,n}^{(k)}(\ell) \leq C^2\al(k+1).
\end{align}
Denote
\begin{align}\label{def Sigma_om}
    \Sg_\om:=\sum_{k=0}^\infty \sum_{\ell=0}^k\bt_{\om,0}^{(k)}(\ell).
\end{align}
Hence, by the assumption of the summability of the $\al(k)$ coming from \eqref{C4}, we have $\Sg_\om$ is well defined and 
\begin{align}\label{eq: assum B}
    0\leq  \Sg_\om \leq \sum_{k=0}^\infty C^2\al(k+1) = C^2\al <\infty.
\end{align}
Thus we have that  
\begin{align}\label{eq: theta def}
    \ta_\om(s):=1-\sum_{k=0}^\infty \hat{q}_{\om,0}^{(k)}(s)
    =1-(1-e^{is})\sum_{k=0}^\infty\sum_{\ell=0}^k e^{i\ell s}\bt_{\om,0}^{(k)}(\ell),
\end{align}
and furthermore we have shown that $|\ta_\om(s)|$ is bounded uniformly in $\om$ and $s$.

\begin{remark}
    Note that the $\bt_{\om,0}^{(k)}(0)$ are the $\hat q_{\om,0}^{(k)}$ from \cite{AFGTV-TFPF}, and thus the standard extremal index $\ta_{\om,0}$ from \cite{AFGTV-TFPF} is given by 
    \begin{align}\label{def usual EI}
        \ta_{\om,0}:=1-\sum_{k=0}^\infty \bt_{\om,0}^{(k)}(0).
    \end{align}
    
\end{remark}
\begin{remark}
    From the calculation of \eqref{eq: beta bound} and \eqref{eq: assum B}, we note that the assumption \eqref{B} can be weakened from the uniform requirement that $\|\ind_{H_{\om,n}}\|_{\cB_\om}\leq C$. All that is required is that 
    $$\sum_{k=0}^\infty\|\ind_{H_{\sg^{-(k+1)}\om,n}}\|_{\cB_{\sg^{-(k+1)}\om}}\al(k+1)<\infty$$ 
    for each $n\in\NN$ and $m$-a.e. $\oio$.
\end{remark}

Under our assumptions \eqref{C1}--\eqref{C8} we are able to apply Theorem 2.1.2 in \cite{AFGTV-TFPF} to get the following result.

\begin{theorem}\label{thm: dynamics perturb thm}
    Suppose that  $(\mathlist{\bcomma}{\Om, m, \sg, \cJ, T, \cB, \cL_0, \nu_0, \phi_0, H_n})$ is a random perturbed system. 
	If \eqref{C8} holds, then for $m$-a.e. $\om\in\Om$ 
	\begin{align*}
		\lim_{n\to\infty}
		\frac{\lm_{\om,0}-\lm_{\om,n,s}}{\lm_{\om,0}\mu_{\om,0}(H_{\om,n})}
		=(1-e^{is})\ta_{\om}(s).
	\end{align*}
\end{theorem}

Now following the proof of \cite[Theorem 2.4.5]{AFGTV-TFPF} and making the obvious, minor changes to suit the perturbations of our current setting, assuming \eqref{C1}--\eqref{C8} and \eqref{S}, we obtain the following.
\begin{theorem}\label{thm CF}
    Given the random perturbed system $(\mathlist{\bcomma}{\Om, m, \sg, \cJ, T, \cB, \cL_0, \nu_0, \phi_0, H_n})$ that also satisfies assumptions \eqref{C8} and \eqref{S}, we have that for each $s\in \mathbb R$ and $m$-a.e. $\oio$
    \begin{align}
		\label{evtthmeqn}
		\lim_{n\to\infty}\nu_{\om,0}\lt(e^{is S_{\om,n,n}}\rt)
		=
		\lim_{n\to\infty}\mu_{\om,0}\lt(e^{is S_{\om,n,n}}\rt)
		=
		\lim_{n\to\infty}\frac{\lm_{\om,n,s}^n}{\lm_{\om,0}^n}
		=
		\exp\left(-(1-e^{is})\int_\Om t_\om\ta_\om(s)\, dm(\om)\right).    
    \end{align}
\end{theorem}
Denote the right hand side of \eqref{evtthmeqn} by 
    \begin{align}
    \label{eq:phidefn}
        \vp(s):=
        \begin{cases}
        1 & \text{if } s=0,\\
        \exp\left(-(1-e^{is})\int_\Om t_\om\ta_\om(s)\, dm(\om)\right)
        & \text{if } s\neq 0.
        \end{cases}
    \end{align}
    Note that since $|\ta_\om(s)|$ is bounded above and below uniformly in $\om$ and $s$, and since $t\in L^\infty(m)$, we have that $|\int_\Om t_\om\ta_\om(s)\, dm(\om)|<\infty$, which further implies that $\vp(s)$ is continuous at $s=0$.

From Theorem \ref{thm CF}, we get that $\vp(s)$ (which does not depend on $\om$) is the limit of the characteristic  functions of the ($\om$-dependent) random variables
$$
Z_{\om,n}:=
S_{\om,n,n}(x)=\sum_{j=0}^{n-1}{\ind}_{H_{\sg^j\om,n}}(T_\om^jx).
$$
Since $\vp(s)$ is continuous at $s=0$, it follows from the L\'evy Continuity Theorem (see \cite[Theorem 3.6.1]{lukacs}) that $\vp$ is the characteristic function of an $\om$-independent random variable $Z$ on some probability space $(\Gm', \mathcal{B}', \mathbb{P}')$ to which the sequence of random variables $(Z_{\om,n})_{n=1}^\infty$ (for $m$-a.e. $\om\in\Om$) converge in distribution. 

We will denote the distribution of $Z$ by $\nu_Z$.  
The random variable $Z$ is non-negative and integer valued 
as it is the distributional limit of a sequence of integer-valued random variable. 
Furthermore, for $m$-a.e. $\om\in\Om$ we have that 
$$
\mathbb{P}'(Z=k)=\nu_{Z}(\{k\})=\lim_{n\rightarrow \infty}\mu_{\om,0}(Z_{\om,n}=k).
$$
If no confusion arises, we will denote the underlying probability with $\mathbb{P},$ instead of $\mathbb{P}',$  and its moments with  $\mathbb{E},$ $\text{Var},$ etc., which are actually computed with the distribution $\nu_Z.$

Since $Z$ is clearly infinitely divisibility, i.e. for each $N\in\NN$ we have    
\begin{align*}
        \vp(s)&= \exp\left(-(1-e^{is})\int_\Om t_\om\ta_\om(s)\, dm(\om)\right)
        =\lt(\exp\left(-\frac{1-e^{is}}{N}\int_\Om t_\om\ta_\om(s)\, dm(\om)\right)\rt)^N,
\end{align*}
by \cite[Section 12.2]{fellerv1} (see also \cite[Theorem 3.1]{Zhang2016}) we have that $Z$ is a compound Poisson random variable. 
 
    \begin{remark}
        Using the L\'evy inversion formula we can calculate the probability mass function associated to the random variable $Z$, that is for each $k\in\{0,1,2,\dots\}$ we have 
        \begin{align*}
            \PP(Z=k) = \lim_{T\to\infty}\frac{1}{2T}\int_{-T}^Te^{-isk}\exp\left(-(1-e^{is})\int_\Om t_\om\ta_\om(s)\, dm(\om)\right) ds.
        \end{align*}
    \end{remark}

\begin{remark}\label{rem: moments}
        In view of \eqref{eq:phidefn} we see that the map $s\mapsto \vp(s)$ is $N$-times differentiable at $s=0$ if the map $s\mapsto\ta_\om(s)$ is $N$-times differentiable at $s=0$. It is well known that if $N$ is even, then the RV $Z$ has finite moments up to order $N$, and if $N$ is odd, then $Z$ has finite moments up to order $N-1$ (see, for example, \cite[Theorem 2.3.3]{lukacs}). Since the differentiability of $\ta_\om(s)$ implies the differentiability of $\vp$, in this case we have that 
        $D_s^k\vp(0)=i^k\EE(Z^k)$ for $k\leq N$, where $D_s$ denotes the derivative operator with respect to $s$. In light of (\ref{eq: theta def}) and the fact that $D_s^N e^{i\ell s}=(i\ell)^Ne^{i\ell s}$, it is clear that the map $s\mapsto\ta_\om(s)$ is $N$-times differentiable at $s=0$ if
        \begin{align}\label{eq: sum check}
            \sum_{k=0}^\infty \sum_{\ell=0}^k\ell^N\bt_{\om,0}^{(k)}(\ell) <\infty.
        \end{align} 
        Using \eqref{eq: assum B}, we see that (\ref{eq: sum check}) will hold if 
        $$\sum_{k=0}^\infty \alpha(k+1)\sum_{\ell=0}^k \ell^N<\infty.$$
        Note that \eqref{eq: sum check} certainly holds if the decay rate $\al(k)$ in assumption \eqref{C4} is exponential, i.e. if $\al(k)=\vkp^k$ for some $\vkp\in(0,1)$. In fact, if we have exponential decay in \eqref{C4}, then the map $s\mapsto\ta_\om(s)$ is infinitely differentiable at $s=0$, and thus, each of the moments of the RV $Z$ is exists and is finite. 
    \end{remark}
    The next proposition follows from Remark \ref{rem: moments} and elementary calculations. 
    \begin{proposition}\label{prop modes}
        If the map $s\mapsto\ta_\om(s)$ is twice differentiable (with respect to $s$) at $s=0$, then, if $Z$ is the compound Poisson RV corresponding to the CF $\vp$ coming from Theorem \ref{thm CF}, we have 
    \begin{align*}
        \EE(Z) = \int_\Om t_\om\, dm(\om)
        \qquad\text{ and }\qquad
        \Var(Z) = \int_\Om t_\om\lt(1+2\Sg_\om\rt)\, dm(\om).
    \end{align*}
    \end{proposition}
    \begin{proof}
    Note that if $s\mapsto\ta_\om(s)$ is twice differentiable, then so is the map $s\mapsto\vp(s)$. Thus we must have that $D_s\vp(0)=i\EE(Z)$ and $D_s^2\vp(0)=i^2\EE(Z^2)=-\EE(Z^2)$.
    Let $u(s)=(e^{is}-1)\int t_\om\ta_\om(s)$, then $\vp(s)=e^{u(s)}$ by (\ref{eq:phidefn}).  Then we have 
    $D_s\vp(s)=D_su(s)\cdot \vp(s)$ and $D_s^2\vp(s)=\lt((D_su(s))^2+D_s^2u(s)\rt)\cdot \vp(s)$.
        Calculating $D_su(s)$ gives
        $$
        D_su(s)
        =
        ie^{is}\int_\Om t_\om\ta_\om(s)\, dm(\om)
        +
        (e^{is}-1)\int_\Om t_\om D_s\ta_\om(s)\, dm(\om),
        $$
        where (by (\ref{eq: theta def}))
        $$ D_s\ta_\om(s)=ie^{is}\sum_{k=0}^\infty\sum_{j=0}^k e^{i\ell s}\bt_{\om,0}^{(k)}(\ell)
        +(e^{is}-1)\sum_{k=0}^\infty\sum_{\ell=0}^k i\ell e^{i\ell s}\bt_{\om,0}^{(k)}(\ell).
        $$
        In view of $\Sg_\om$ defined in \eqref{def Sigma_om}, we note that  
        $$
        D_s\ta_\om(0)=i\sum_{k=0}^\infty\sum_{j=0}^k \bt_{\om,0}^{(k)}(\ell)
        =i\Sg_\om.
        $$
        Since $\vp(0)=1$, $\ta_\om(0)=1$, and $|D_s\ta_\om(0)|<\infty$ (by \eqref{eq: assum B}) we have that 
        $$
        D_s\vp(0)=D_su(0)
        =i\int_\Om t_\om\, dm(\om),
        $$
        and so $\EE(Z)=\int_\Om t_\om\, dm(\om)$.
        
        Now calculating $D_s^2u(s)$ we have 
        $$
        D_s^2u(s)=
        -e^{is}\int_\Om t_\om\ta_\om(s)\, dm(\om)
        +2ie^{is}\int_\Om t_\om D_s\ta_\om(s)\, dm(\om)
        +(e^{is}-1)\int_\Om t_\om D_s^2\ta_\om(s)\, dm(\om).
        $$
        Noting that $|D_s^2\ta_\om(0)|$ is finite and $\ta_\om(0)=1$, plugging in $s=0$ gives that 
        \begin{align*}
        D_s^2\vp(0)        &=(D_su(0))^2+D_s^2u(0)
        \\
        &=i^2\lt(\int_\Om t_\om\, dm(\om)\rt)^2+i^2\int_\Om t_\om\, dm(\om)+2i\int_\Om t_\om D_s\ta_\om(0)\, dm(\om)
        \\
        &=-\lt(\int_\Om t_\om\, dm(\om)\rt)^2-\int_\Om t_\om\, dm(\om)-2\int_\Om t_\om \Sg_\om\, dm(\om).
        \end{align*}
        Since $D_s^2\vp(0)=-\EE(Z^2)$, we have
        \begin{align}\label{eq E(Z2)}
            \EE(Z^2)=\lt(\int_\Om t_\om\, dm(\om)\rt)^2+\int_\Om t_\om\, dm(\om)+2\int_\Om t_\om \Sg_\om\, dm(\om).
        \end{align}
        Calculating the variance $\Var(Z)=\EE(Z^2)-(\EE(Z))^2$ finishes the proof. 
    \end{proof}
Since $Z$ a compound Poisson random variable there exists $N$, a Poisson RZ with parameter $\vta$, and iid RVs $X_j$ such that  
\begin{equation*}
    Z:=\sum_{j=1}^N X_j.
\end{equation*}
Furthermore, the CF $\vp(s)=\phi_Z(s)$ of the RV $Z$ is related to the CF $\phi_{X_i}(s)$ of the RV $X_i$ by the following
\begin{align}\label{CF Z}
    \phi_Z(s)=e^{\vta(\phi_{X_1}(s)-1)}.
\end{align}
Arguing as in Section \ref{sec deterministic}, for each $\ell\geq 1$, we set $S_\ell=\sum_{i=1}^\ell X_i$ and $S_0$ the RV equal to $0$. Since $Z$ is compound Poisson we can write
\begin{equation}\label{eq Z dist}
    \mathbb{P}(Z=k)=\sum_{\ell=0}^{\infty}\mathbb{P}(N=\ell)\mathbb{P}(S_\ell=k),
\end{equation}
and so, using Wald's equation and Proposition \ref{prop modes}, we have that $\mathbb{E}(Z)=\vta\mathbb{E}(X_1)=\int_\Om t_\om\, dm(\om)$. Therefore, we must have that 
\begin{align}\label{eq lambda}
    \vta=\frac{\int_\Om t_\om\, dm(\om)}{\EE(X_1)}.
\end{align}
Now we note that it follows from the definition of the RV $Z$ that $\PP(Z=0)$ is given by Gumbel's law (see \cite[Section 2.4]{AFGTV-TFPF}), and thus 
\begin{align}\label{eq Prob Z=0}
    \mathbb{P}(Z=0)=e^{-\int_\Om t_\om\ta_{\om,0}\, dm(\om)},
\end{align}
where $\ta_{\om,0}$ is the usual extremal index defined by \eqref{def usual EI}.
Setting $k=0$ in \eqref{eq Z dist}, the only term which survives is $\ell=0$ (since $\PP(S_0=0)=1$ as $S_0\equiv 0$), which gives
\begin{align}\label{eq Prob Z N}
    \mathbb{P}(Z=0)=\mathbb{P}(N=0)=e^{-\vta}=e^{-\frac{\int_\Om t_\om\, dm(\om)}{\EE(X_1)}},
\end{align}
and therefore using \eqref{eq lambda}, \eqref{eq Prob Z=0}, and \eqref{eq Prob Z N} to solve for $\EE(X_1)$ we get 
\begin{align}\label{eq E(X)}
    \mathbb{E}(X_1)=\frac{\int_\Om t_\om\, dm(\om)}{\int_\Om t_\om\ta_{\om,0}\, dm(\om)}.
\end{align}
Note that this relationship generalizes the relationship found in Section \ref{sec deterministic} and in \cite{HV20}.
It follows from \eqref{eq lambda} and \eqref{eq E(X)} that 
\begin{align}\label{vartheta}
    \vta = \int_\Om t_\om\ta_{\om,0}\, dm(\om).    
\end{align}

To find the variance of the RV $X_i$ we differentiate the CF $\phi_Z(s)$ and $\phi_{X_1}(s)$ as in Proposition \ref{prop modes}. In particular, we have 
$$
D_s\phi_Z(s)=\vta D_s\phi_{X_1}(s)\cdot\phi_Z(s)
$$ 
and 
$$
D_s^2\phi_Z(s)=\lt( \vta D_s^2\phi_{X_1}(s)+(\vta D_s\phi_{X_1}(s))^2\rt)\cdot\phi_Z(s).
$$ 
Plugging in $s=0$ together with \eqref{eq lambda}, fact that $D_s\phi_{X_1}(0)=i\EE(X_1)$, and \eqref{eq E(X)} gives that 
\begin{align*}
    -\EE(Z^2)=D_s^2\phi_Z(0) &=  D_s^2\phi_{X_1}(0)\int_\Om t_\om\ta_{\om,0}\, dm(\om) +\lt(i\int_\Om t_\om\, dm(\om)\rt)^2
    \\
    &=D_s^2\phi_{X_1}(0)\int_\Om t_\om\ta_{\om,0}\, dm(\om) -\lt(\int_\Om t_\om\, dm(\om)\rt)^2.
\end{align*}
So using \eqref{eq E(Z2)}, the formula for $\EE(Z^2)$ from the proof of Proposition \ref{prop modes}, we can solve for $\EE(X_1^2)=-D_s^2\phi_{X_1}(0)$ to get 
\begin{align*}
    \EE(X_1^2) = \frac{\int_\Om t_\om (1+2\Sg_\om)\, dm(\om)}{\int_\Om t_\om\ta_{\om,0}\, dm(\om)}, 
\end{align*}
and so we can calculate the variance of $X_1$ to be 
\begin{align*}
    \Var(X_1) &=
    \frac{\int_\Om t_\om (1+2\Sg_\om)\, dm(\om)}{\int_\Om t_\om\ta_{\om,0}\, dm(\om)}
    -
    \lt(\frac{\int_\Om t_\om\, dm(\om)}{\int_\Om t_\om\ta_{\om,0}\, dm(\om)}\rt)^2.
\end{align*}
Using \eqref{eq:phidefn}, \eqref{CF Z}, \eqref{vartheta}, and the fact that $\vp(s)=\phi_Z(s)$, we can solve for the CF $\phi_{X_1}(s)$ of $X_1$ to get 
\begin{align*}
    \phi_{X_1}(s)=(e^{is}-1)\frac{\int_\Om t_\om \ta_\om(s)\, dm(\om)}{\int_\Om t_\om\ta_{\om,0}\, dm(\om)} +1. 
\end{align*}
Thus, we have just proved the following proposition. 
\begin{proposition}
     The CF of the compound Poisson RV $Z$ can be written as 
    \begin{align*}
        \vp(s)=\phi_Z(s)=e^{\vta(\phi_{X_1}(s)-1)},
    \end{align*}
    where $Z=\sum_{k=1}^N X_k$, $N$ is a Poisson RV with parameter $\vta = \int_\Om t_\om\ta_{\om,0}\, dm(\om)$, and $X_1, X_2, \dots$ are  RV whose CF is given by 
    \begin{align*}
        \phi_{X_1}(s)=(e^{is}-1)\frac{\int_\Om t_\om \ta_\om(s)\, dm(\om)}{\int_\Om t_\om\ta_{\om,0}\, dm(\om)} +1. 
    \end{align*}
    Moreover, if the map $s\mapsto\ta_\om(s)$ is twice differentiable (with respect to $s$) at $s=0$, then,
    the expectation and variance of the RV $X_j$ is given by 
    \begin{align*}
        \mathbb{E}(X_1)=\frac{\int_\Om t_\om\, dm(\om)}{\int_\Om t_\om\ta_{\om,0}\, dm(\om)}
        \quad\text{ and }\quad
        \Var(X_1) &=
    \frac{\int_\Om t_\om (1+2\Sg_\om)\, dm(\om)}{\int_\Om t_\om\ta_{\om,0}\, dm(\om)}
    -
    \lt(\frac{\int_\Om t_\om\, dm(\om)}{\int_\Om t_\om\ta_{\om,0}\, dm(\om)}\rt)^2.
    \end{align*}
    
\end{proposition}

The distribution of $X_1$ can be obtained by applying the L\'evy inversion formula
\begin{align}\label{eq prob X inv rand}
    \PP(X_1=k)=\lim_{T\to\infty}\frac{1}{2T}\int_{-T}^T e^{-isk}\lt((e^{is}-1)\frac{\int_\Om t_\om \ta_\om(s)\, dm(\om)}{\int_\Om t_\om\ta_{\om,0}\, dm(\om)} +1\rt)\, ds.
\end{align}
Following the procedure outlined in Section \ref{sec deterministic} to solve for $\PP(X_1)$ by differentiating the PGF $G_Z(s)$ of $Z$, and noting that $\vta=\int_\Om t_\om\ta_{\om,0}\, dm(\om)$, we have 
\begin{align}\label{eq pgf prob Z rand}
    \PP(Z=K)
    &=
    \frac {e^{-\int_\Om t_\om\ta_{\om,0}\, dm(\om)}}{K!}\lt(\sum_{k=0}^K\frac{1}{k!}D_s^K(\~g_{X_1}^k)(0)\rt).
\end{align}
Thus one could obtain the distribution for $X_1$ (as an alternative to using \eqref{eq prob X inv rand}) by using \eqref{eq pgf prob Z rand} to solve for $\PP(X_1=k)=\sfrac{D_s^k(\~g_{X_1})(0)}{k!}$ recursively in terms of $\PP(Z=k)$.

\section{Checking our general assumptions for a large class of examples}\label{sec: existence}

In this section we present an explicit class of random piecewise-monotonic interval maps for which our general assumptions \eqref{C1}-\eqref{C7} hold. We follow the approach of Section 2.5 of \cite{AFGTV-TFPF} adapted to the setting of our current perturbations.

We now suppose that the spaces $\cJ_{\om}=[0,1]$ for each $\om\in\Om$ and the maps $T_\om:[0,1]\to[0,1]$ are surjective, finitely-branched, piecewise monotone, nonsingular (with respect to Lebesgue measure $\Leb$), and that there exists $C\geq 1$ such that 
\begin{align}
\tag{F1}\label{F1}
\esssup_\om |T_\om'|\leq C
\qquad\text{ and }\qquad
\esssup_\om d(T_\om)
\leq C,
\end{align}
where $d(T_\om):=\sup_{y\in[0,1]}\# T_\om^{-1}(y)$.
We let $\cZ_{\om,0}$ denote the (finite) monotonicity partition of $T_\om$ and for each $k\geq 2$ we let $\cZ_{\om,0}^{(k)}$ denote the partition of monotonicity of $T_\om^k$.
\begin{enumerate}[align=left,leftmargin=*,labelsep=\parindent]
\item[\mylabel{MC}{MC}]
The map $\sg:\Om\to\Om$ is a homeomorphism, the skew-product map $T:\Om\times [0,1]\to\Om\times [0,1]$ is measurable, and $\omega\mapsto T_\omega$ has countable range. 
\end{enumerate}
\begin{remark}\label{M2 holds}
Under assumption (\ref{MC}), the family of transfer operator cocycles $\{\cL_{\om,n,s}\}$ satisfies the conditions of Theorem 17 \cite{FLQ2} ($m$-continuity and $\sigma$ a homeomorphism). Note that condition \eqref{MC}
implies that $T$ satisfies \eqref{M1} and the cocycle generated by $\cL_0$ satisfies condition \eqref{M2}.
\end{remark}
Recall that the \emph{variation} of $f:[0,1]\to\mathbb{R}_+$ on $Z\subset [0,1]$ be
\begin{align*}
\var_{Z}(f)=\underset{x_0<\dots <x_k, \, x_j\in Z}{\sup}\sum_{j=0}^{k-1}\absval{f(x_{j+1})-f(x_j)}, 
\end{align*}
and  $\var(f):=\var_{[0,1]}(f)$. We let $\BV=\BV([0,1])$ denote the set of complex-valued functions on $[0,1]$ that have bounded variation. 
We define the Banach space $\BV_{1}\sub L^\infty(\Leb)$ to be the set of (equivalence classes of) functions of bounded variation on $[0,1]$, with norm given by
\begin{align*}
\|f\|_{\BV_1}:=\underset{\tilde{f}=f \ \Leb\text{ a.e.}}{\inf} \var(\tilde{f})+\Leb(|f|).    
\end{align*} 
We denote the supremum norm on $L^\infty(\Leb)$ by $\|\spot\|_{\infty,1}$. 
We set $J_\om:=|T_\om'|$ and define the random Perron--Frobenius operator, acting on functions in $BV$
$$
P_\om(f)(x):=\sum_{y\in T_\om^{-1}(x)}\frac{f(y)}{J_\om(y)}.
$$
The operator $P$ satisfies the well-known property that 
\begin{align}\label{PF prop}
\int_{[0,1]}P_\om(f)\,d\Leb=\int_{[0,1]}f\,d\Leb
\end{align}
for $m$-a.e. $\om\in\Om$ and all $f\in\BV$.
Recall from Section~\ref{sec: random setting} that $g_0=\set{g_{\om,0}}_{\om\in\Om}$ and that
$$\cL_{\om,0}(f)(x):=\sum_{y\in T_\om^{-1}(x)}g_{\om,0}(y)f(y),\quad f\in\BV.$$
We assume that the weight function $g_{\om,0}$ lies in $\BV$ for each $\om\in\Om$ and satisfies
\begin{equation}\tag{F2}\label{F2}
\esssup_\omega \| g_{\omega,0}\|_{\infty,1}<\infty,
\end{equation}
and 
\begin{equation}
\tag{F3}\label{F3}
\essinf_\omega \inf g_{\omega,0}>0.
\end{equation}
Note that \eqref{F1} and \eqref{F2} together imply 
\begin{align}\label{fin sup L1}
\esssup_\om\norm{\cL_{\om,0}\ind}_{\infty,1}\leq \esssup_\om d(T_\om)\norm{g_{\om,0}}_{\infty,1}<\infty
\end{align}
and 
\begin{align}\label{fin gJ}
\esssup_\om\norm{g_{\om,0}J_\om}_{\infty,1}<\infty.
\end{align}
We also assume a uniform covering condition\footnote{We could replace the covering condition with the assumption of a strongly contracting potential. See \cite{AFGTV-IVC} for details.}
:
\begin{enumerate}[align=left,leftmargin=*,labelsep=\parindent]
\item[\mylabel{F4}{F4}]
For every subinterval $J\subset [0,1]$ there is a $k=k(J)$ such that for a.e.\ $\omega$ one has $T^k_\omega(J)=[0,1]$.
\end{enumerate}
Concerning the holes system we assume that $H_{\om,n}\sub [0,1]$ are chosen so that assumption
\eqref{A} 
holds. We also assume for each $\om\in\Om$ and each $n\geq 1$ that $H_{\om,n}$ is composed of a finite union of intervals such that 
\begin{enumerate}[align=left,leftmargin=*,labelsep=\parindent]
\item[\mylabel{F5}{F5}]
There is a uniform-in-$n$ and uniform-in-$\omega$ upper bound $\frak{h}\geq 1$ on the number of connected components of $H_{\om,n}$.
\end{enumerate}
We now assume that
\begin{align}
\tag{F6}\label{F6}
\lim_{n\to \infty}\esssup_\om \Leb(H_{\om,n})=0.
\end{align}
Because we are considering small \jaadd[r]{target}s shrinking to zero measure, it is natural to  assume that
\begin{align}
\tag{F7}\label{F7}
T_\om(H_{\om,n}^c)=[0,1]
\end{align}
for $m$-a.e. $\om\in\Om$ and all $n\geq 1$ sufficiently large.
Further we suppose that there exists $N'\geq1$ such that
\footnote{Note that the coefficients appearing in \eqref{F8} are not optimal. See  \cite{AFGTV20} and \cite{AFGTV-TFPF} for how this assumption may be improved. }
\begin{equation}
\tag{F8}\label{F8}
(9+12\frak{h}N')\cdot\esssup_\om\|g_{\omega,0}^{(N')}\|_{\infty,1}
<
\essinf_\om\inf\cL_{\om,0}^{N'}\ind.
\end{equation}
\begin{remark}
    Note that if $\esssup_\om\|g_{\om,0}\|_{\infty,1}<\essinf_\om\inf\cL_{\om,0}\ind$, then the gap between $\esssup_\om\|g_{\om,0}\|_{\infty,1}$ and $\essinf_\om\inf\cL_{\om,0}\ind$ grows exponentially, and thus such an $N'$ exists.
\end{remark}
For each $k\in\NN$ and $\om\in\Om$ we let $\sA_{\om,0}^{(k)}$ be the collection of all finite partitions of $[0,1]$ such that
\begin{align}\label{eq: def A partition}
\var_{A_i}(g_{\om,0}^{(k)})\leq 2\norm{g_{\om,0}^{(k)}}_{\infty,1}
\end{align}
for each $\cA=\set{A_i}\in\sA_{\om,0}^{(k)}$.
Given $\cA\in\sA_{\om,0}^{(k)}$, let $\cZ_{\om,*}^{(k)}$ be the coarsest partition amongst all those finer than $\cA$ and $\cZ_{\om,0}^{(k)}$. 
\begin{remark}
Note that if $\var_Z(g^{(k)}_{\om,0})\leq 2\norm{g_{\om,0}^{(k)}}_{\infty,1}$ for each $Z\in\cZ_{\om,0}^{(k)}$ then we can take the partition $\cA=\cZ_{\om,0}^{(k)}$. Furthermore, the $2$ above can be replaced by some $\hat\al\geq 0$ (depending on $g_{\om,0}$) following the techniques of \cite{AFGTV20} and \cite{AFGTV-TFPF}. 
\end{remark}

We assume the following covering condition 
\begin{enumerate}[align=left,leftmargin=*,labelsep=\parindent]
\item[\mylabel{F9}{F9}]
There exists $k_o(N')\in\NN$ such that for $m$-a.e. $\om\in\Om$, and all $Z\in\cZ_{\om,*}^{(N')}$ we have $T_\om^{k_o(N')}(Z)=[0,1]$, where $N'$ is the number coming from \eqref{F8}.
\end{enumerate}
\begin{remark}
Note that the uniform covering time assumption \eqref{F9} clearly holds if \eqref{F4} holds and if there are only finitely many maps $T_\om$. Remark 2.C.2 of \cite{AFGTV-TFPF} presents an alternative assumption to \eqref{F9} which could be adapted to the setting of this paper.
\end{remark}

Since $T_\om^j\rvert_Z$ is one-to-one for each $Z\in\cZ_{\om,*}^{(k)}$ and each $1\leq j\leq k$, the number of connected components of $Z\cap T_\om^{-j}(H_{\sg^j\om})$ is bounded above by $\frak{h}$ by assumption \eqref{F5}. Since  the function $\exp(is \ind_{H_{\sg^j\om,n}}\circ T_\om^j)$ has at most $2\frak{h}$ jump discontinuities on each interval $Z\in\cZ_{\om,*}^{(k)}$ for each $1\leq j\leq k$, we have that $e^{is S_{\om,n,k}}$ has at most $2\frak{h}k$ many jump discontinuities. Furthermore, $e^{is S_{\om,n,k}}$ is constant between consecutive discontinuities.  Using this together with the fact that the distance in $\mathbb{C}$ between points on the unit circle is bounded above by 2, we have that
\begin{align*}
    \var_Z(e^{is S_{\om,n,k}})
    \leq 
    4\frak{h}k
\end{align*}
for each $Z\in\cZ_{\om,*}^{(k)}$.
Recalling that $g_{\om,n,s}^{(k)}:=\exp(S_k(\vp_{\om,0})+isS_{\om,n,k})=g_{\om,0}\exp(isS_{\om,n,k})$, where $g_{\om,0}=\exp(\vp_{\om,0})$ is as in Section~\ref{sec: random setting}, we see that \eqref{eq: def A partition} implies that 
\begin{align}
    \var_{Z}(g_{\om,n,s}^{(k)})&\leq 
    \norm{g_{\om,0}^{(k)}}_{\infty,1}\var_Z(e^{is S_{\om,n,k}})+\|e^{is S_{\om,n,k}}\|_\infty\var_Z(g_{\om,0}^{(k)})
    \nonumber\\
    &\leq 
    4\frak{h}k\|g_{\om,0}^{(k)}\|_{\infty,1}+2\norm{g_{\om,0}^{(k)}}_{\infty,1}
    \leq
    \|g_{\om,0}^{(k)}\|_{\infty,1}\lt(2+4\frak{h}k\rt)
    \label{eq: def A partition for g_ep}
\end{align}
for each $Z\in \cZ_{\om,*}^{(k)}$.

The following lemma of \cite{AFGTV-TFPF} extends several results in \cite{DFGTV18A} from the specific weight $g_{\omega,0}=1/|T'_\omega|$ to general weights satisfying the conditions just outlined.

\begin{lemma}[Lemma 2.5.9 of \cite{AFGTV-TFPF}]
\label{DFGTV18Alemma}
Assume that a family of random piecewise-monotonic interval maps $\{T_{\omega}\}$ satisfies \eqref{F1}, (\ref{F2}), (\ref{F3}), and (\ref{F4}), as well as (\ref{F8}) and \eqref{MC} for the unperturbed operator cocycle.
Then \eqref{CCM} and 
\eqref{C1} hold, and the parts of \eqref{C2}, \eqref{C3}, \eqref{C4}, \eqref{C5}, and  \eqref{C7} corresponding to the unperturbed operator cocycle hold as well. 
Further, $\nu_{\om,0}$ is fully supported and condition \eqref{C4} holds with $C_f=K$, for some $K<\infty$, and $\alpha(N)=\gamma^N$ for some $\gamma<1$.
\end{lemma}

From Lemma~\ref{DFGTV18Alemma} we have that $\lm_{\om,0}:=\nu_{\sg\om,0}(\cL_{\om,0}\ind)$ and thus for each $k\geq 1$ we have 
\begin{align}\label{uniflbL0}
\essinf_\om\lm_{\om,0}^k
\geq \essinf_\om\inf\cL_{\om,0}^k\ind
\geq \essinf_\om\inf g_{\om,0}^{(k)}>0.
\end{align}
Furthermore, since $\nu_{\om,0}$ is fully supported on $[0,1]$ and non-atomic for a.e. $\om$, we have that \eqref{F6} implies that for $m$-a.e. $\om\in\Om$  $\lim_{n\to\infty}\nu_{\om,0}(H_{\om,n})=0$, and since $\esssup_\om\lm_{\om,0}$ is bounded (by \eqref{fin sup L1}) together with Remark \ref{rem check C6}, we have that \eqref{C6} holds. 

We now use hyperbolicity of the unperturbed transfer operator cocycle to guarantee that we have hyperbolic cocycles for large $n$, which will yield \eqref{C2}, \eqref{C3}, \eqref{C4}, \eqref{C5}, \eqref{C6}, and \eqref{C7} for large $n$.

The main result of this section is the following.
\begin{lemma}\label{existence lemma}
    Assume that a family of random piecewise-monotonic interval maps $\{T_{\omega}\}$ satisfies \eqref{F1}--\eqref{F9} together with \eqref{MC}. Then \eqref{CCM}, \eqref{A}, and \eqref{B} hold as well as \eqref{C1}. Furthermore, for all $n\geq 1$ sufficiently large, conditions \eqref{C2}--\eqref{C7} hold. In particular $(\mathlist{\bcomma}{\Om, m, \sg, [0,1], T, \BV_1, \cL_0, \nu_0, \phi_0, H_n})$ is a random perturbed system.
\end{lemma}
\begin{proof}
The proof of this lemma follows similarly to the proof 2.5.10 in \cite{AFGTV-TFPF}. 
The main argument is to apply Theorem A \cite{C19}. Denoting $\hat\cL_{\om,n,s}:=\lm_{\om,0}^{-1}\cL_{\om,n,s}$ for each $\om$, $n\ge 1$, and $s\in\RR\bs\{0\}$ we can write the hypotheses of Theorem A \cite{C19}, in our notation, as:
\begin{enumerate}
\item\label{harry assum 1 hat} $\hat\cL_{\omega,0}$ is a hyperbolic transfer operator cocycle on $\BV_1$ with norm $\|\cdot\|_{\BV_1}$ and a one-dimensional leading Oseledets space (see Definition 3.1 \cite{C19}), and slow and fast growth rates $0<\gamma<\Gamma$, respectively. We will construct $\gamma$ and $\Gamma$ shortly.
\item\label{harry assum 2 hat} The family of cocycles $\{\hat\cL_{\omega,n,s}\}_{n\ge n_0}$ satisfy a uniform Lasota--Yorke inequality 
\begin{equation*}
\|\hat\cL^k_{\omega,n,s}f\|_{\BV_1}\le A\alpha^k\|f\|_{\BV_1}+B^k\|f\|_{1}
\end{equation*}
for a.e.\ $\omega$ and $n \ge n_0$, where $\alpha\le \gamma<\Gamma\le B$.
\item\label{harry assum 3 hat} $\lim_{n\to \infty}\esssup_{\omega}\trinorm{\hat\cL_{\omega,0}-\hat\cL_{\omega,n,s}}= 0$, where $\vertiii{\spot}$ is the $\BV-L^1(\Leb)$ triple norm.
\end{enumerate} 
The proofs of items \eqref{harry assum 1 hat} and \eqref{harry assum 3 hat} are similar to the proofs given in the proof of Lemma 2.5.10 in \cite{AFGTV-TFPF}. To prove item \eqref{harry assum 2 hat}, we require the following the following Lasota--Yorke inequality inspired by Lemma 1.5.1 of \cite{AFGTV-TFPF}.
\begin{lemma}\label{closed ly ineq App} 
For any $f\in\BV$ we have 
\begin{align*}
\var(\cL_{\om,n,s}^k(f))
\leq
\lt(9+12\frak{h}k\rt)\norm{g_{\om,0}^{(k)}}_{\infty,1}\var(f)+
\frac{\lt(8+12\frak{h}k\rt)\norm{g_{\om,0}^{(k)}}_{\infty,1}}{\min_{Z\in\cZ_{\om,*}^{(k)}}\nu_{\om,0}(Z)}\nu_{\om,0}(|f|).
\end{align*}
\end{lemma}
\begin{proof}
By considering intervals $Z$ in $\cZ_{\om,*}^{(k)}$, we are able to write 
\begin{align}\label{eq: ly ineq 1}
\cL_{\om,n,s}^kf=\sum_{Z\in\cZ_{\om,*}^{(k)}}(\ind_Z f g_{\om,n,s}^{(k)})\circ T_{\om,Z}^{-k}
\end{align} 
where 
$$	
T_{\om,Z}^{-k}:T_\om^k(I_{\om,n,s})\to Z
$$ 
is the inverse branch that takes $T_\om^k(x)$ to $x$ for each $x\in Z$. Now, since 
$$
\ind_Z\circ T_{\om,Z}^{-k}=\ind_{T_\om^k(Z)},
$$
we can rewrite \eqref{eq: ly ineq 1} as 
\begin{align}\label{eq: closed ly ineq 2}
\cL_{\om,n,s}^kf=\sum_{Z\in\cZ_{\om,*}^{(k)}}\ind_{T_\om^k(Z)} \lt((f g_{\om,n,s}^{(k)})\circ T_{\om,Z}^{-k}\rt).
\end{align}
So,
\begin{align}\label{closed var tr op sum}
\var(\cL_{\om,n,s}^kf)\leq \sum_{Z\in\cZ_{\om,*}^{(k)}}\var\lt(\ind_{T_\om^k(Z)} \lt((f g_{\om,n,s}^{(k)})\circ T_{\om,Z}^{-k}\rt)\rt).
\end{align}
Now for each $Z\in\cZ_{\om,*}^{(k)}$, using \eqref{eq: def A partition for g_ep}, we have 
\begin{align}
&\var\lt(\ind_{T_\om^k(Z)} \lt((f g_{\om,n,s}^{(k)})\circ T_{\om,Z}^{-k}\rt)\rt)
\leq \var_Z(f g_{\om,n,s}^{(k)})+2\sup_Z\absval{f g_{\om,n,s}^{(k)}}
\nonumber\\
&\qquad\qquad\leq 3\var_Z(f g_{\om,n,s}^{(k)})+2\inf_Z\absval{f g_{\om,n,s}^{(k)}}
\nonumber\\
&\qquad\qquad\leq 3\norm{g_{\om,n,s}^{(k)}}_{\infty,1}\var_Z(f)+3\sup_Z|f|\var_Z(g_{\om,n,s}^{(k)})+2\norm{g_{\om,n,s}^{(k)}}_{\infty,1}\inf_Z|f|
\nonumber\\
&\qquad\qquad\leq 
3\norm{g_{\om,0}^{(k)}}_{\infty,1}\var_Z(f)+\lt(6+12\frak{h}k\rt)\norm{g_{\om,0}^{(k)}}_{\infty,1}\sup_Z|f|+2\norm{g_{\om,0}^{(k)}}_{\infty,1}\inf_Z|f|
\nonumber\\
&\qquad\qquad\leq 
\lt(9+12\frak{h}k\rt)\norm{g_{\om,0}^{(k)}}_{\infty,1}\var_Z(f)+\lt(8+12\frak{h}k\rt)\norm{g_{\om,0}^{(k)}}_{\infty,1}\inf_Z|f|
\nonumber\\
&\qquad\qquad\leq
\lt(9+12\frak{h}k\rt)\norm{g_{\om,0}^{(k)}}_{\infty,1}\var_Z(f)+\lt(8+12\frak{h}k\rt)\norm{g_{\om,0}^{(k)}}_{\infty,1}\frac{\nu_{\om,0}(|f\rvert_Z|)}{\nu_{\om,0}(Z)}.
\label{closed var ineq over partition}
\end{align}
Using \eqref{closed var ineq over partition}, we may further estimate \eqref{closed var tr op sum} as
\begin{align}
\var(\cL_{\om,n,s}^kf)
&\leq
\sum_{Z\in\cZ_{\om,*}^{(k)}} \lt(\lt(9+12\frak{h}k\rt)\norm{g_{\om,0}^{(k)}}_{\infty,1}\var_Z(f)+\lt(8+12\frak{h}k\rt))\norm{g_{\om,0}^{(k)}}_{\infty,1}\frac{\nu_{\om,0}(|f\rvert_Z|)}{\nu_{\om,0}(Z)}\rt)
\nonumber\\
&\leq 
\lt(9+12\frak{h}k\rt)\norm{g_{\om,0}^{(k)}}_{\infty,1}\var(f)+
\frac{\lt(8+12\frak{h}k\rt)\norm{g_{\om,0}^{(k)}}_{\infty,1}}{\min_{Z\in\cZ_{\om,*}^{(k)}}\nu_{\om,0}(Z)}\nu_{\om,0}(|f|),
\label{ORLYLY}
\end{align}
and thus the proof of Lemma \ref{closed ly ineq App} is complete. 
\end{proof}

To complete the proof of item \eqref{harry assum 2 hat} of Lemma \ref{existence lemma} we begin by setting 
\begin{align}\label{al def}
\alpha^{N'}:=\frac{(9+12\frak{h}N'))\esssup_\om \|g_{\om,0}^{( N')}\|_{\infty,1}}{\essinf_\om\inf\mathcal{L}^{ N'}_{\om,0}\ind}
<1,
\end{align}
which is possible by  (\ref{F8}). 
For item \eqref{harry assum 2 hat} using the the Lasota--Yorke inequality proven in Lemma \ref{closed ly ineq App}, together with \eqref{al def} and dividing through by $\lm_{\om,0}^{N'}$ gives
\begin{align}
\var(\hat\cL_{\omega,n, s}^{N'} f)
&\le 
\frac{(9+12\frak{h}N'))\norm{g_{\om,0}^{(k)}}_{\infty,1}}{\lm_{\om,0}^{N'}}\var(f)
+\frac{(8+12\frak{h}N'))\norm{g_{\om,0}^{(k)}}_{\infty,1}}{\lm_{\om,0}^{N'}\min_{Z\in \mathcal{Z}_{\omega,*}^{(N')}(\mathcal{A})}\nu_{\omega,0}(Z)}\nu_{\omega,0}(|f|)
\nonumber\\
&\le 
\alpha^{N'}\var(f)+\frac{\al^{N'}}{\min_{Z\in \mathcal{Z}_{\omega,*}^{(N')}(\mathcal{A})}\nu_{\omega,0}(Z)}\nu_{\omega,0}(|f|).
\label{LYineqorly hat}
\end{align} 
We note that \eqref{F3}, \eqref{F9}, \eqref{uniflbL0}, and the equivariance of the backward adjoint cocycle imply that for $Z\in\cZ_{\om,*}^{(N')}$ we have
\begin{align}\label{LY LB calc}
\nu_{\om,0}(Z)
=
\nu_{\sg^{k_o(N')}\om,0}\left(\left(\lm_{\om,0}^{k_o(N')}\right)^{-1}\cL_{\om,0}^{k_o(N')}\ind_Z\right)
\geq
\frac{\inf g_{\om,0}^{k_o(N')}}{\lm_{\om,0}^{k_o(N')}}>0.
\end{align}
As stated above, the remainder of the proof of Lemma \ref{existence lemma} follows similarly to the proof of Lemma 2.5.10 in \cite{AFGTV-TFPF}.
\end{proof}

\section{Examples}\label{S:examples}

In this section we present several examples of random dynamics resulting in varying types of compound Poisson distributions. We provide examples of systems in the class described in Section \ref{sec: existence}.

\begin{example}\label{ex 3}\textbf{Compound Poisson That is Not P\'olya-Aeppli From Random Maps and Random \jaadd[r]{Target}s:}
Let $\{\Omega_j\}$ be a partition of $\Omega$. We consider the following family of maps $\{T_\omega\}$:
\begin{equation}
	\label{eg1}
	T_\omega(x)=\left\{
	\begin{array}{ll}
		L_\om(x), & \hbox{$0\le x\le (1-1/\gm_\omega)/2$;} 
        \\
		\gm_\omega x-(\gm_\omega-1)/2, & \hbox{$(1-1/\gm_\omega)/2\le x\le (1+1/\gm_\omega)/2$;} 
        \\ R_\om(x), & \hbox{$(1+1/\gm_\omega)/2\le x\le 1$,}
	\end{array}
	\right.
\end{equation}
where $L|_{\Om_j}$ and $R|_{\Om_j}$ are maps with finitely many full linear branches,  
$1<\gm\le \gm_\omega\le \Gm<\infty$ and $\gm\rvert_{\Om_j}$ is constant   
for each $j\geq 1$.
We take $g_{\omega,0}=1/|T'_\omega|$ and since these maps have full linear branches, Lebesgue measure is preserved by all maps $T_\omega$;  i.e.\ $\mu_{\omega,0}=\mathrm{Leb}$ for a.e. $\omega$.
The central branch has slope $\gm_\omega$ and passes through the fixed point $x_0=1/2$, which lies at the center of the central branch;  see Figure \ref{fig:map}.
\begin{figure}[hbt]
	\centering
	\includegraphics[width=6cm]{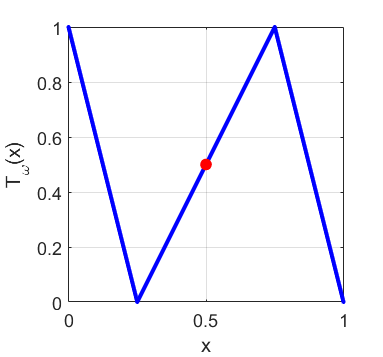}
	\caption{Graph of a map $T_\omega$, with $\gm_\omega=2$.}\label{fig:map}
\end{figure}
A specific random driving could be $\sigma:S^1\to S^1$ given by $\sigma(\omega)=\omega+\alpha$ for some $\alpha\notin\mathbb{Q}$ and $\gm_\omega=\gm^0+\gm^1\cdot\omega$ for $1< \gm^0<\infty$ and $0<\gm^1<\infty$, but only the ergodicity of $\sigma$ will be important for us.
We consider the Banach spaces $\cB_\om$ to be the space $\BV$ of complex-valued functions of bounded variation.

To verify our general assumptions \eqref{M1}, \eqref{M2}, \eqref{CCM}, \eqref{A}, \eqref{B}, \eqref{C1} - \eqref{C7} it suffices to check the assumptions \eqref{F1}--\eqref{F9}
\footnote{We note that since the \jaadd[r]{target}s consist of a single interval, to to check \eqref{F8} with $N'=1$, it suffices to ensure that $\essinf_\om\gm_\om>21$ and that the slopes of the branches of the maps $L_\om$ and $R_\om$ are at least as large as $\gm_\om$.}
This is done in a similar manner to the way in which the assumptions (E1)--(E9) of \cite[Section 2.5]{AFGTV-TFPF} are checked in \cite[Example 2.7.1]{AFGTV-TFPF} for transfer operators acting on the space of real-valued BV functions with different perturbations. 
Therefore we have only to verify assumptions \eqref{C8} and \eqref{S}.

For each $\oio$ and each $\nin$ consider a sequence of shrinking balls $H_{\om,n}$ containing the fixed point $x_0$
such that $x_0\in H_{\om,n+1}\sub H_{\om,n}$ for each $\nin$. Then assumption \eqref{B} is clearly satisfied. Furthermore the \jaadd[r]{target}s $H_{\om,n}$ are chosen so that the assumption \eqref{S} is satisfied. Such \jaadd[r]{target}s $H_{\om,n}$ are constructed in Section 2.7 of \cite{AFGTV-TFPF} as neighborhoods of extremal points of observation functions $h:M\to\RR$. 

Assumption \eqref{S} states that $\mu_{\om,0}(H_{\om,n})\leq \frac{|t|_\infty+W}{n}$.
In this example, all fibre measures $\mu_{\om,0}$ are Lebesgue and so the \jaadd[r]{target}s $H_{\om,n}$ are neighbourhoods of $x_0$ of diameter no greater than $(|t|_\infty+W)/n$.
Because the maps $T_\omega$ are locally expanding about $x_0$, for fixed $k>0$ one can find a large enough $n$ so that it is impossible to leave a small neighbourhood of $x_0$ and return after $k$ iterations.
Therefore we can always find $n$ sufficiently large to guarantee that $\bt_{\om,n}^{(k)}(\ell)=0$ for all $0<\ell<k$.
If this were not the case, there must be a positive $\mu_{\sg^{-(k+1)}\om,0}$-measure set of points that (i) lie in $H_{\sg^{-(k+1)}\om,n}$, (ii)  land outside $\ell<k$ of the \jaadd[r]{target}s $H_{\sg^{-k}\om,n},\ldots,H_{\sg^{-1}\om,n}$ for the next $k$ iterations, and then (iii) land in $H_{\om,n}$ on the $k+1^{\mathrm{st}}$ iteration, which we have argued is impossible.
On the other hand, the equality $\ell=k$ corresponds to the situation where points remain in the sequence of \jaadd[r]{target}s for \emph{all} $k$ iterations, which will occur for a positive measure set of points.
This immediately implies that $\bt_{\om,n}^{(k)}(k)\neq 0$, and thus
\begin{align*}
    \hat{q}_{\om,n}^{(k)}(s)
    &=
    (1-e^{is})e^{ik s}\bt_{\om,n}^{(k)}(k)
    =(1-e^{is})e^{ik s}
    \frac{\mu_{\sg^{-(k+1)}\om,0}\lt(\bigcap_{j=0}^{k+1}T_{\sg^{-(k+1)}\om}^{-j}(H_{\sg^{-(k+1)+j}\om,n})\rt)}{\mu_{\om,n}(H_{\om,n})}.
\end{align*}
Because the scaling $t_\om$ varies along the orbit $\sg^{-(k+1)}\omega,\ldots,\omega$, to simplify the above expression we assume that $1<\esssup_\om t_\om/t_{\sigma^{-1}\om}<\essinf_\om |T'_\om(x_0)|$.
This mild assumption on the variation of the scaling $t$ removes the ``Case 2'' considerations in Example 1 \cite{AFGTV-TFPF}, and taking the limit as $n\to\infty$ gives 
\begin{align*}
    \hat q_{\om,0}^{(k)}(s) 
    &=
    (1-e^{is})\frac{e^{ik s}}{\lt|DT_{\sg^{-(k+1)}\om}^{k+1}(x_0)\rt|}
    =
    (1-e^{is})\frac{e^{ik s}}{\prod_{j=0}^k \gm_{\sg^{-(k+1)+j}}}, 
\end{align*}

and thus we have 
\begin{align*}
    \ta_{\om}(s) 
    = 
    1-(1-e^{is})\sum_{k=0}^\infty \frac{e^{ik s}}{\lt|DT_{\sg^{-(k+1)}\om}^{k+1}(x_0)\rt|}.
\end{align*}
Applying Theorem \ref{thm CF} gives that 
\begin{align*}
    \vp(s)
    &= \exp\lt(-(1-e^{is})\int_\Om t_\om\lt(1-(1-e^{is})\sum_{k=0}^\infty \frac{e^{ik s}}{\lt|DT_{\sg^{-(k+1)}\om}^{k+1}(x_0)\rt|}\rt)\, dm(\om)\rt)
    \\
    &= \exp\lt(-(1-e^{is})\int_\Om t_\om\lt(1-(1-e^{is})\sum_{k=0}^\infty \frac{e^{ik s}}{\prod_{j=0}^k\gm_{\sg^{-(k+1)+j}\om}}\rt)\, dm(\om)\rt).
    \end{align*}
    Note that since $1<\gm\leq\gm_\om\leq \Gm$, the series $\sum_{k=0}^\infty \frac{e^{ik s}}{\prod_{j=0}^k\gm_{\sg^{-(k+1)+j}\om}}$ converges absolutely to say $\Sg_\om(s)$ (which depends on $\om$). Since $\Sg_\om(s)$ is not necessarily a geometric series, then $\vp(s)$ is the characteristic function of a compound Poisson distribution which is not the P\'olya-Aeppli distribution. It is unclear which specific compound Poisson distribution is represented by $\vp(s)$.

    Note that the tail of the series $\Sg_\om(s)$ is approximately given by $e^{-is}\sum_{k=N_\om}^\infty (e^{is -\int_\Om\log\gm_\om\, dm})^{k+1}$ where $N_\om$ is the ($\om$-dependent) time that it takes for the Birkhoff Ergodic Theorem to apply.

\end{example}

\begin{example}\label{ex 2}\textbf{P\'olya-Aeppli From Non-Random Slope and Random \jaadd[r]{Target}s:}
  Taking the family of maps $\{T_\om\}$ from the previous example, we now suppose that the slope of the central branch is constant; i.e.\ $DT_\om(x_0)=\gm_\om\equiv \gm>1$ is constant.
    Then we have 
    \begin{align*}
        \vp(s)
    &= \exp\lt(-(1-e^{is})\int_\Om t_\om\lt(1-(1-e^{is})\sum_{k=0}^\infty \frac{e^{ik s}}{\gm^{k+1}}\rt)\, dm(\om)\rt)
    \\
    &=
    \exp\lt(-\lt(1-\gamma^{-1}\rt)\left(\int_\Om t_\om\, dm(\om)\right)\lt(\frac{1-e^{is}}{1-e^{is}\gm^{-1}}\rt)\rt).
    \end{align*}
Then $\vp(s)$ is the characteristic function of a P\'olya-Aeppli (geometric Poisson) distributed random variable $Z$ with parameters given by $\rho=\gm^{-1}\in (0,1)$ and $\vta = (1-\gm^{-1})\int_\Om t_\om dm$. 
\end{example}

\begin{example}\label{ex 2.1}\textbf{P\'olya-Aeppli From I.I.D.\  Maps and \jaadd[r]{Target}s:}
    Now suppose that we are in the setting of the previous example where now we have two maps $T_1$ and $T_2$ chosen iid with probabilities $p_1$ and $p_2$ respectively. Again we let $x_0$ denote the common fixed point of the maps $T_1$ and $T_2$ and let $\gm_i$ denote the slope of the map $T_i$ at $x_0$ for $i=1,2$. In this case $m$ is Bernoulli. Further suppose there are two values of $t$, namely $t_1$ and $t_2$, which are chosen iid with probabilities $\vrho_1$ and $\vrho_2$ respectively\footnote{The argument here goes through exactly the same if $t$ takes on countably many values $t_i$ with probability $\vrho_i$.}. 
 
    Thus, using the fact that the $\gm$ are chosen iid, we have that
    \begin{align}\label{eq iid prod gm}
        \int_\Om  \prod_{j=0}^k\gm_{\sg^{-(k+1)+j}\om}^{-1}\, dm(\om)
        =
        \lt(\frac{p_1}{\gm_1}+\frac{p_2}{\gm_2}\rt)^{k+1}
        =
        \lt(\frac{p_1\gm_2+p_2\gm_1}{\gm_1\gm_2}\rt)^{k+1}
        =:\zt^{k+1}.
    \end{align}
    Note that $\zt\in(0,1)$.
    Now since $t_\om$ is independent of the $\gm$ terms in the product appearing in \eqref{eq iid prod gm} (since the product does not contain a $\gm_\om$), we can use \eqref{eq iid prod gm} to write that 
    \begin{align*}
        &\int_\Om t_\om\lt(1-(1-e^{is})\sum_{k=0}^\infty e^{ik s}\prod_{j=0}^k\gm_{\sg^{-(k+1)+j}\om}^{-1}\rt)\, dm(\om)
        \\
        &\qquad=\ol{t} \lt(1-(1-e^{is})\sum_{k=0}^\infty e^{ik s}\int_\Om\prod_{j=0}^k\gm_{\sg^{-(k+1)+j}\om}^{-1}\, dm(\om)\rt)
        \\
        &\qquad=\ol{t} \lt(1-(1-e^{is})\sum_{k=0}^\infty e^{ik s}\zt^{k+1}\rt)
        =\ol{t}\lt(\frac{1-\zt}{1-e^{is}\zt}\rt),
    \end{align*}
    where $\ol{t}=\int_\Om t_\om\, dm = t_1\vrho_1+t_2\vrho_2$.
    Inserting this into the formula for $\vp(s)$ gives 
    \begin{align*}
        \vp(s)
        =\exp\lt(-\ol{t}(1-e^{is})\lt(\frac{1-\zt}{1-e^{is}\zt}\rt)\rt)
        =\exp\lt(-\ol{t}(1-\zt)\lt(\frac{1-e^{is}}{1- e^{is}\zt}\rt)\rt),
    \end{align*}
    which is the characteristic function of a P\'olya-Aeppli ($\rho=\zt\in (0,1)$ and $\vta = (1-\zt)\ol{t}$) distributed random variable.
\end{example}

\begin{example}\label{ex 1 standard Poisson}\textbf{Standard Poisson from Random Maps and Random \jaadd[r]{Target}s:}

Unlike the previous examples, we now provide an example in which the characteristic function produced from Theorem \ref{thm CF} is for a standard Poisson random variable rather than compound Poisson. We recall the setting of Example 4 of \cite{AFGTV-TFPF}.
Let $\Omega=\{0,1,2,3\}^{\mathbb{Z}},$ with $\sigma$ the bilateral shift map, and $m$ an invariant ergodic measure.

To each letter $j=0,\dots,l-1$ we associate a rational number $v_j\in(0,1)\cap\QQ$. We consider \jaadd[r]{target}s $H_{\om,n}= B(v_{\omega_0}, e^{-z_n(\om)})$,
where $\omega_0$ denotes the $0$-th  coordinate of $\omega\in \Omega$ and the thresholds $z_n(\om)$ are chosen such that assumption \eqref{S} holds.
For each $\omega\in \Omega$ we associate a map $T_{\omega_0},$ where $T_0,\dots,T_{3}$ are maps  of the circle which we will take as $\beta$-maps of the form
 $T_i(x)=\beta_ix+r\pmod 1$, with $\beta_i\in \mathbb{N}$, $\bt_i\geq 3$, and $0\le r<1$ irrational and independent of $i$. Similar arguments to those given in \cite[Example 4]{AFGTV-TFPF} show that our assumptions \eqref{M1}, \eqref{M2}, \eqref{CCM}, \eqref{A}, \eqref{C1} - \eqref{C7} are satisfied for this system. This is accomplished as in Example \ref{ex 3}, by showing that our assumptions \eqref{F1}--\eqref{F9} follow from arguments similar to those of \cite[Example 4]{AFGTV-TFPF} used to verify the assumptions (E1)--(E9). Furthermore, assumption \eqref{B} is satisfied by taking the Banach spaces $\cB_\om$ again to be $\BV$.

Following the argument of \cite[Example 4]{AFGTV-TFPF}, it follows that a necessary condition to get $\bt_{\om,n}^{(k)}(\ell)\neq 0$ for some $k\geq 0$ and $0\leq \ell\leq k$, and thus $\hat{q}_{\om,0}^{(k)}(s)\neq0$, is that the center $v_{(\sigma^{-(k+1)}\om)_0)}$ is sent to the center $v_{\om_0}.$  Let $z$ be one of these rational centers. The iterate $T^n_{\om}(z)$ has the form $T^n_{\om}(z)=\beta_{\om_{n-1}}\cdots \beta_{\om_0}z+k_n r\, \pmod 1$, where $k_n$ is an integer.
Therefore such an iterate will never be a rational number, which shows that all $\hat{q}_{\om,0}^{(k)}(s)=0$ for each $k\ge 0$, $\om$, and each $s\in\RR\bs\{0\}$. Therefore we have $\ta_\om(s)\equiv1$, and so applying Theorem \ref{thm CF}, we have 
\begin{align*}
    \vp(s)
    =
    \exp\lt(-(1-e^{is})\int_\Om t_\om\, dm(\om)\rt),
\end{align*}
which is the  characteristic function of a Poisson random variable with parameter $\vta=\int_\Om t_\om\, dm(\om)$.
\end{example}

\section*{Acknowledgments}
The authors thank the Mathematical Research Institute MATRIX for hosting a workshop during which much of this work was conceived. SV thanks  the support and hospitality of the Sydney Mathematical Research Institute (SMRI), the University of New South Wales, and the University of Queensland.  
The research of SV was supported by the project {\em Dynamics and Information Research Institute} within the agreement between UniCredit Bank and Scuola Normale Superiore di Pisa and by the Laboratoire International Associ\'e LIA LYSM, of the French CNRS and  INdAM (Italy).  SV was also supported by the project MATHAmSud TOMCAT 22-Math-10, N. 49958WH, du french  CNRS and MEAE. 
JA is supported by the ARC Discovery projects DP220102216, and GF and CG-T are partially supported by the ARC Discovery Project DP220102216.
The authors thanks R. Aimino for having shown them the reference \cite{Zhang2020} and L. Amorim and N. Haydn for useful discussions on alternative approaches to get compound Poisson statistics.

\bibliographystyle{abbrv}
\bibliography{poisson}

\end{document}